\documentclass{amsart}      
\usepackage{amssymb,amsmath,graphicx}
\usepackage{enumerate}
\usepackage{xcolor}
\newtoks\prt 

\numberwithin{equation}{section}

\newtheorem{thm}{Theorem}[section]

\newtheorem{ques}[thm]{Question} 
\newtheorem{lemma}[thm]{Lemma} 
\newtheorem{prop}[thm]{Proposition} 
\newtheorem{cor}[thm]{Corollary}

\newtheorem{example}[thm]{Example}

\theoremstyle{definition} 
\newtheorem{example2}[thm]{Example} 
\newtheorem{remark}[thm]{Remark}

\newtheorem{remarks}[thm]{Remarks}

\def\eqn#1$$#2$${\begin{equation}\label#1#2\end{equation}}

\def\F{\mathcal F}

\def\ce{\mathbb C}

\def\diam{\operatorname{diam}} 
 
\def\ep{\varepsilon} 
 
\def\en{\mathbb N} 
\def\er{\mathbb R} 
 
\def\ef{\mathbb F} 
\def\zet{\mathbb Z} 
 
\def\dist{\operatorname{dist}}

\def \sign{\operatorname{sign}} 
\def\ov{\overline}

\def \Lip {\operatorname{Lip}}

\def \ext {\operatorname{ext}}

\def\span{\operatorname{span}} 
 
\def\osc{\operatorname{osc}}

\def \reg {\partial _{\kern1pt\text{reg}}} 
 
\def\iff{\Longleftrightarrow}

\def\ip#1#2{\left\langle#1,#2\right\rangle}

\def\dh{\widehat{\operatorname{d}}}
\def\clu#1{\operatorname{clust}_{w^*}(#1)}
\def\wde#1{\widetilde{\delta}\left(#1\right)}
\def\de#1{\delta\left(#1\right)}
\def\cj#1{\operatorname{\varepsilon_J}\left(#1\right)}
\def\dcj#1{\widetilde{\operatorname{\varepsilon_J}}\left(#1\right)}
\def\cjr#1{\operatorname{\varepsilon_J^R}\left(#1\right)}
\def\dcjr#1{\widetilde{\operatorname{\varepsilon_J^R}}\left(#1\right)}

\newcommand{\norm}[1]{\left\|#1\right\|}

\renewcommand{\Re}{\operatorname{Re}}

\newcommand{\ca}[2][]{\operatorname{ca}_{#1}\left(#2\right)}
\newcommand{\wca}[2][]{\widetilde{\operatorname{ca}}_{#1}\left(#2\right)}

\newcommand{\wk}[2][]{\operatorname{wk}_{#1}\left(#2\right)}
\newcommand{\wck}[2][]{\operatorname{wck}_{#1}\left(#2\right)}
\newcommand{\wscl}[1]{\overline{#1}^{w^*}}

\newcommand{\abs}[1]{\left|#1\right|}
\newcommand{\setsep}{;\,}

\title{Quantitative Schur property and measures of weak non-compactness}
\author{Ond\v{r}ej F.K. Kalenda }
\address{Ondřej F.K. Kalenda\\
Charles University\\
Faculty of Mathematics and Physics\\
Department of Mathematical Analysis \\
Sokolovsk\'{a} 83, 186 \ 75\\Praha 8, Czech Republic}
\email{kalenda@karlin.mff.cuni.cz}



\thanks{Supported by the Research grant GA\v{C}R 23-04776S}

\begin{document}

\begin{abstract}
     We compare several versions of the quantitative Schur property of Banach spaces. We establish their equivalence up to multiplicative constants and provide examples clarifying when the change of constants is necessary. We also give exact results on preservation of the quantitative Schur property by finite or infinite direct sums. We further prove a sufficient condition for the $1$-Schur property which simplifies and generalizes previous results. 
     We study in more detail relationship of the  quantitative Schur property to quantitative weak sequential completeness and to equivalence of measures of weak non-compactness. We also illustrate the difference of real and complex settings. To this end we prove and use the optimal version of complex quantiative Rosenthal $\ell_1$-theorem. Finally, we give two examples of Lipschitz-free spaces over countable graphs which have quantitative Schur property, but not the $1$-Schur property.
    
\end{abstract}

\maketitle

\section{Introduction}

The quantitative approach to Banach space theory provides more precise versions of classical theorems and a deeper insight into geometric properties of Banach spaces. For example, in \cite{f-krein,Gr-krein,CMR} three different proofs of a quantitative version of Krein's theorems are given, there are quantitative version of James' compactness theorem \cite{CKS,Gr-James}, of Rosenthal's $\ell_1$-theorem \cite{behrends}, of Eberlein-\v{S}mulyan theorem \cite{AC-jmaa}. Further, quantitative versions of many properties have been studied, this applies in particular to weak sequential completeness \cite{wesecom}, Schur property and its variants \cite{qschur,qschur-dp,chen-posschur}, Dunford-Pettis property and its variants \cite{qdpp,p-DP}, reciprocal Dunford-Pettis property \cite{rdpp}, Pe\l{}czy\'nski property (V) and its variants \cite{haaanja-(V),haaanja-Cstar,chen-vstar},
Grothendieck property \cite{haaanja-gr,lechner-1gr}, Banach-Saks property including higher-order versions \cite{kryczka-jmaa12,krzyczka-jmaa13,BKS-bs,silber-qbs} etc.

In many of the quoted results measures of non-compactness play a key role. Let us recall basic related quantities. Assume that $X$ is a Banach space and $A,B$ are two nonempty subsets of $X$. We set
$$\dh(A,B)=\sup\{\dist(a,B)\setsep a\in A\}.$$
This quantity is sometimes called \emph{the excess $A$ from $B$}. If $A$ is bounded, $\dh(A,B)$ is necessarily finite. The excess is used to define \emph{Hausdorff measure of non-compactness} by the formula
$$\chi(A)=\inf\{\dh(A,F)\setsep F\subset X\mbox{ finite}\}$$
for any bounded set $A\subset X$. Similarly, \emph{De Blasi measure of weak non-compactness} is defined by
$$\omega(A)=\inf\{\dh(A,K)\setsep K\subset X\mbox{ weakly compact}\},\quad A\subset X\mbox{ bounded}.$$
Another measure of weak non-compactness inspired by the Banach-Alaoglu theorem is
$$\wk{A}=\dh(\wscl{A},X),\quad A\subset X\mbox{ bounded},$$
where the weak$^*$-closure is considered in the bidual $X^{**}$ ($X$ is considered canonically as a subspace of $X^{**}$).

We note that $\chi$ is indeed a measure of non-compactness in the sense that $\chi(A)=0$ if and only if $A$ is relatively compact. Similarly we have
$$\omega(A)=0 \iff \wk{A}=0 \iff A\mbox{ is relatively weakly compact}.$$
Quantity $\omega$ was introduced by De Blasi \cite{deblasi} in order to prove a fixed-point theorem. Quantity $\wk{\cdot}$ was used to establish the above-mentioned quantitative versions od Krein's theorem, Eberlein-\v{S}mulyan theorem and James' compactness theorem.

It is easy to observe that always $\wk{A}\le \omega(A)$. These two quantities are not equivalent in general (this follows from \cite[Corollary 3.4 and Theorem 2.3]{AC-meas}, an easy explicit example can be found in \cite[Example 10.1]{qdpp}) but in many classical spaces they are equal. This applies in particular to $c_0(\Gamma)$ and $L^1(\mu)$ \cite[Proposition 10.2 and Theorem 7.5]{qdpp}, preduals of von Neumann algebras or, more generally, of JBW$^*$-triples \cite{mwnc}, spaces of nuclear operators $N(\ell_p(I),\ell_q(J))$ for $p,q\in(1,\infty)$ \cite{HK-nuclear}. In fact it is still an open question whether there is a `classical' Banach space in which the two quantities are not equivalent.

Coincidence of measures of weak non-compactness in $L^1(\mu)$ was used in \cite[Theorem 8.1]{qdpp} to establish the `direct quantitative Dunford-Pettis property' for $L^1$-preduals, in particular for $C(K)$-spaces. Connections of equivalence of measures of weak non-compactness with quantitative Schur and Dunford-Pettis properties are shown in \cite[Theorem 2.1]{qschur-dp}. 

In the present paper we investigate in more detail quantitative Schur property and its relationship to measures of weak non-compactness. The paper is organized as follows.

In Section~\ref{sec:prel} we collect basic quantities measuring behavior of bounded sequences in Banach spaces
and recall some of their properties. We also provide definitions and basic facts on quantitative Schur property and quantitative weak sequential completeness. We further include a precise version of a lemma relating real and complex setting.

In Section~\ref{sec:rosenthal} we focus on measuring weak non-precompactness using a quantity from \cite{kania-qgp}. We analyze its connections to the quantitative version of Rosenthal's $\ell_1$-theorem from \cite{behrends} and differences between the real and complex cases. We also provide the optimal version of Rosenthal's theorem for complex spaces.

Section~\ref{sec:wsc} is devoted to measuring weak non-compactness in weakly sequentially complete spaces. In these spaces the measure of weak non-precompactness from Section~\ref{sec:rosenthal} becomes a measure of weak non-compactness. We compare it to the remaining measures and establish a connection to quantitative weak sequential completeness.

In Section~\ref{sec:schur} we deal with quantitative Schur property from \cite{qschur}. We establish precise connections of different versions of this property, in particular to the $1$-strong Schur property from \cite{Go-Ka-Li}, including the difference of the real and complex setting. We also analyze connections of the quantitative Schur property to measures of weak non-compactness.

Section~\ref{sec:m1} contains a sufficient condition for the $1$-Schur property. It provides an alternative
proof and a generalization of the main result of \cite{qschur-dp}.

In Section~\ref{sec:preserving} we address stability of quantitative Schur property to products. We prove two
positive results and show limits of such preservation.

In Section~\ref{sec:LF} we focus on the Schur property in Lipschitz-free spaces. It was recently characterized in \cite{p1u}, we are concern with possibility of quantitative strengthening of this result. In particular, we present two example of a Lipschitz-free spaces over graphs. The first one fails the $1$-Schur property but enjoys the $1$-strong Schur property. This witnesses that these two properties are different also within real Banach spaces. The second one fails even the $1$-strong Schur property. 

In the last section we point out some open problems. 

\section{Preliminaries -- quantities related to sequences}\label{sec:prel}

In this section we collect several quantities measuring behavior of bounded sequences in Banach spaces 
with focus on their applications to measures of (weak) non-compactness. 
We start by saying that we deal both with real and complex Banach spaces and that we use $\ef$ to denote the respective field ($\er$ or $\ce$). We continue by quantities related to norm-convergence.

Let $X$ be a Banach space over $\ef$ and let $(x_n)$ be a bounded sequence in $X$.
We set
$$\begin{aligned}
\ca{x_n} & = \osc(x_n) = \inf_{n\in\en} \diam\{x_k\setsep k\ge n\},\\
\wca{x_n} & = \inf\{ \ca{x_{n_k}}\setsep (n_k)\mbox{ increasing}\}.
\end{aligned}$$
The first quantity is the oscillation of the sequence and serves as a measure of non-cauchyness. Clearly $\ca{x_n}=0$ if and only if $(x_n)$ is  norm-Cauchy, hence norm-convergent. The second quantity measures oscillation of subsequences. It may be used to define an alternative measure of non-compactness by
$$\beta(A)=\sup\{\wca{x_n}\setsep (x_n)\mbox{ is a sequence in }A\}, \quad A\subset X\mbox{ bounded}.$$
This quantity coincides with Istr\u{a}\c{t}escu's measure of non-compactness from \cite{istratescu}, i.e.,
$$\beta(A)=\inf\{c>0\setsep A \mbox{ contains no infinite $c$-discrete subset}\}.$$
Indeed, if $(x_n)$ is $c$-discrete, clearly $\wca{x_n}\ge c$, hence inequality `$\ge$' holds. To prove the converse, take any $c<\beta(A)$. Then there is a sequence $(x_n)$ in $A$ with $\wca{x_n}>c$. It follows that no subsequence of $(x_n)$ can have diameter at most $c$. We use the classical Ramsey theorem to find a
$c$-discrete subsequence of $(x_n)$, which completes the argument. We note that $\beta$ is indeed a measure of non-compactness and it is equivalent to the Hausdorff measure $\chi$. More precisely, we have
$$\chi(A)\le \beta(A)\le 2\chi(A).$$
 
The above quantities have their analogues for the weak topology. More concretely, for a bounded sequence $(x_n)$
in a Banach space $X$ we set
$$\begin{aligned}
\de{x_n} & = \sup \{\ca{x^*(x_n)}\setsep x^*\in B_{X^*}\},\\
\wde{x_n} & = \inf\{ \de{x_{n_k}}\setsep (n_k)\mbox{ increasing}\}.
\end{aligned}$$
The first quantity is the usual measure of weak non-cauchyness. It is zero if and only if $(x_n)$ is weakly Cauchy. 
It can be alternatively expressed as
$$\de{x_n}=\diam \clu{x_n},$$
where $\clu{x_n}$ denotes the set of all cluster points of $(x_n)$ in $(X^{**},w^*)$. These quantities are not directly related with measures of weak non-compactness, but the second will be used to measure weak non-precompactness in the following section. 

Quantity $\de{\cdot}$ and its comparison with other quantities may be used to define quantitative versions of weak sequential completeness and of the Schur property. Following \cite{qschur} we say that a Banach space $X$
\begin{itemize}
    \item  is \emph{$c$-weakly sequentially complete} (shortly \emph{$c$-wsc}), where $c\ge 0$ if 
    $$\dh(\clu{x_n},X)\le c \de{x_n}$$
    for each bounded sequence $(x_n)$ in $X$;
    \item has the \emph{$c$-Schur property} if
    $$\ca{x_n}\le c\de{x_n}$$
     for each bounded sequence $(x_n)$ in $X$.
\end{itemize}
We note that the $c$-Schur property implies $c$-weak sequential completeness by \cite[Proposition 1.1]{qschur}, but quantitative weak sequential completeness accompanied by the Schur property does not imply quantitative Schur property by \cite[Example 1.4]{qschur}. Reflexive spaces are clearly $0$-wsc, $L$-embedded Banach spaces are $\frac12$-wsc by \cite[Theorem 1]{wesecom}. Further, if $X$ is $c$-wsc for some $C<\frac12$, it is reflexive by \cite[Proposition 1.2]{qschur}.

The set of weak$^*$-cluster points of a bounded sequence may be used to define an alternative measure
of weak non-compactness inspired by the Eberlein-\v{S}mulyan theorem using the formula
$$\wck{A}=\sup\{\dist(\clu{x_n},X)\setsep (x_n)\mbox{ is a sequence in }A\},$$
where $A$ is again a bounded subset of a Banach space $X$. This quantity is a measure of weak noncompactness equivalent to $\wk{\cdot}$ as, due to \cite{AC-jmaa}, we have
\begin{equation}\label{eq:wck-wk}
    \wck{A}\le\wk{A}\le 2\wck{A}.
\end{equation}

As we remarked at the beginning of this section, we work simultaneously in real and complex Banach spaces. In fact,
if $X$ is a complex Banach space and $X_R$ is its real variant (i.e., the same space, we just forget multiplication by imaginary numbers), all quantities considered in this section coincide in $X$ and in $X_R$. This is explained in \cite[Section 5]{wesecom}, see also \cite[Section 2.1]{qdpp}.

We finish this section by a lemma on complex numbers which is useful to relate certain quantities which differ in real and complex spaces.

\begin{lemma}\label{L:complex}
    Let $J$ be a finite set and let $\lambda_j$, $j\in J$, be complex numbers. Then there is $I\subset J$ such that
    $$\abs{\sum_{j\in I} \lambda_j}\ge\frac1\pi\sum_{j\in J}\abs{\lambda_j}.$$
    Moreover, the constant $\frac1\pi$ is optimal.
\end{lemma}

\begin{proof} The existence of $I$ is proved in \cite[Lemma 6.3]{rudin}. Let us prove the optimality.

Fix $n\in\en$ and set $J=\{e^{i j\pi/n}\setsep j\in \{0,1,\dots,2n-1\}\}$. Then clearly
$$\sum_{\lambda\in J} \abs{\lambda}=2n.$$
Let $I\subset J$ be such that $\abs{\sum_{\lambda\in I}\lambda}$ is maximal. Observe that
\begin{equation}
    \label{eq:eta}
\forall \eta\in J\colon I\mbox{ contains exactly one of the numbers }\eta,-\eta.\end{equation}
Indeed, fix $\eta\in J$ and set $I^\prime=I\setminus \{\eta,-\eta\}$.  Let 
$$\mu=\sum_{\lambda\in I^\prime}\lambda.$$
Then clearly
$$\mu=\sum_{\lambda\in I^\prime\cup\{\eta,-\eta\}}\lambda$$
as well. We deduce that
$$\abs{\mu+\eta}^2+\abs{\mu-\eta}^2=2(\abs{\mu}^2+\abs{\eta}^2)>2\abs{\mu}^2.$$
Thus
$$\abs{\mu+\eta}>\abs{\mu}\mbox{ or }\abs{\mu-\eta}>\abs{\mu}.$$
The choice of $I$ now yields that it must contain exactly one of the numbers $\eta,-\eta$, which completes the proof of \eqref{eq:eta}.

Next we set 
$$\mu=\sum_{\lambda\in I} \lambda.$$
Up to rotation of the set $I$ we may assume that the argument of $\mu$ belongs to $(\frac\pi2-\frac\pi n,\frac\pi2]$. Then, in particular, both real and imaginary part of $\mu$ are non-negative, hence the maximality assumption easily implies that
$$I\supset\{e^{i j\pi/n}, 0\le j\le \tfrac n2\}.$$
We further claim that
$$e^{ij\pi/n}\in I\mbox{ for } \tfrac n2<j\le n-1.$$
Indeed, assume that $\frac n2<j\le n-1$ and $e^{ij\pi/n}\notin I$. By \eqref{eq:eta} we know that  $-e^{ij\pi/n}\in I$. But then
$$\abs{\mu+2e^{ij\pi/n}}=\abs{\abs{\mu}+2e^{ij\pi/n-\arg \mu}}\ge \abs{\mu}+2\cos(\tfrac{j\pi}{n}-\arg \mu).$$
By our assumption we have
$$\tfrac{j\pi}{n}-\arg \mu< \tfrac{(n-1)\pi}{n}-(\tfrac\pi2-\tfrac\pi n)=\tfrac\pi2$$
and
$$\tfrac{j\pi}{n}-\arg \mu>\tfrac{\pi n/2}{n}-\tfrac\pi2=0,$$
so $\cos(\tfrac{j\pi}{n}-\arg \mu)>0$, hence
$$\abs{\mu+2e^{ij\pi/n}}>\abs{\mu},$$
a contradiction with the maximality assumption. 

We conclude that
$$I=\{e^{ij\pi/n}\setsep j\in\{0,\dots,n-1\}.$$
Hence
$$\abs{\sum_{\lambda\in I}\lambda}=\abs{\sum_{j=0}^{n-1}e^{ij\pi/n}}=\abs{\frac{1-e^{i\pi}}{1-e^{i\pi/n}}}=\frac{2}{\abs{1-e^{i\pi/n}}}=\frac{1/n}{\abs{1-e^{i\pi/n}}}\cdot 2n.$$
Since
$$\frac{1/n}{\abs{1-e^{i\pi/n}}}\to\frac1\pi\mbox{ as }n\to\infty,$$
the optimality of $\frac1\pi$ follows.
\end{proof}

\section{Measuring weak non-precompactness}\label{sec:rosenthal}

Recall that a subset $A$ of a Banach space $X$ is called \emph{weakly precompact} if any sequence in $A$ admits a weakly Cauchy subsequence. To measure weak non-precompactness it is hence natural to define the following quantity (again, $A$ is a bounded subset of a Banach space $X$):
$$\beta_w(A)=\sup\{\wde{x_n}\setsep (x_n)\mbox{ is a sequence in }A\}.$$
This quantity was used implicitly in \cite{rodr-ep-wpc} for real spaces and in \cite{kania-qgp} it is denoted by 
$\operatorname{wpc}_2$. 
It follows by combining the uniform boundedness principle with Rosenthal's $\ell_1$ theorem that a set is weakly precompact if and only if it is bounded and contains no sequence equivalent to the canonical basis of $\ell_1$.
We are going to formulate a quantitative version of this equivalence. To this end we introduce quantites measuring 
 lower $\ell_1$-estimates of bounded sequences. If $(x_n)$ is a bounded sequence in a Banach space $X$, we set
$$\begin{aligned}
 \cj{x_n}&=\sup_{n\in\en} \inf \left\{ \norm{\sum_{k=n}^N\alpha_k x_k} \setsep N\ge n, \alpha_n,\dots,\alpha_N\in\ef, \sum_{k=n}^N\abs{\alpha_k}=1\right\},\\   
 \cjr{x_n}&=\sup_{n\in\en} \inf \left\{ \norm{\sum_{k=n}^N\alpha_k x_k} \setsep N\ge n, \alpha_n,\dots,\alpha_N\in\er, \sum_{k=n}^N\abs{\alpha_k}=1\right\}.
\end{aligned}$$
If we pass to a subsequence, these two quantities cannot decrease. Hence, the subsequence versions will read 
$$\begin{aligned}
\dcj{x_n} & = \sup\{ \cj{x_{n_k}}\setsep (n_k)\mbox{ increasing}\},\\
\dcjr{x_n} & = \sup\{ \cjr{x_{n_k}}\setsep (n_k)\mbox{ increasing}\}.
\end{aligned}$$
These quantities (in the real setting) were introduced by H.~Pfitzner in a preprint version of \cite{pfitzner-inve} (in the published version they are not present anymore). We find them useful to formulate a quantitative version of Rosenthal's $\ell_1$-theorem. Let us point out that, unlike other quantities, they have different behavior in real and complex spaces. We continue by a theorem collecting known connections of these quantities to $\de{\cdot}$ and $\wde{\cdot}$.

\begin{thm}\label{T:rosenthal}
Let $X$ be a Banach space and let $(x_n)$ be a bounded sequence. Then the following inequalities are valid:
\begin{enumerate}[$(a)$]
    \item $\de{x_n}\ge2\cjr{x_n}\ge 2\cj{x_n}$;
    \item $\wde{x_n}\ge2\cjr{x_n}\ge 2\cj{x_n}$;
    \item $\de{x_n}\ge2\dcjr{x_n}\ge 2\dcj{x_n}$;
    \item $\wde{x_n}\le 2 \dcjr{x_n}$;
    \item  $\wde{x_n}\le \pi \dcj{x_n}$ if $X$ is complex.
\end{enumerate}
\end{thm}

\begin{proof}
If $X$ is complex, we denote by $X_R$ its real version. Then quantities $\delta$ and $\widetilde{\delta}$ are the same in $X$ and in $X_R$, as explained in \cite[Section 5]{wesecom}. Therefore the first inequality in $(a)$ follows from \cite[Lemma 5]{wesecom}, the second inequality is trivial. Assertions $(b)$ and $(c)$ follow from $(a)$.

Assertion $(d)$ is a consequence of \cite[Theorem 3.2]{behrends}. Indeed, the quoted result, using our notation, says that, assuming $\dcjr{x_n}<c$, we may find a subsequence $(x_{n_k})$ with $\de{x_{n_k}}\le 2c$.

Assertion $(e)$, with constant $4\sqrt2$ instead of $\pi$, follows from \cite[Theorem 3.3]{behrends}. 
The constant may be improved to $\pi$ by a minor change in the final computation. Indeed, assuming $\sum_{\rho=1}^r\abs{a_\rho}=1$, by Lemma~\ref{L:complex} there is $S\subset \{1,\dots,r\}$ such that $\abs{\sum_{\rho\in S} a_\rho}\ge\frac1\pi$. Therefore it is enough to replace condition `$a_\rho\in S_1$' by `$\rho\in S$' in the final computation.
\end{proof}

\begin{remarks}
(1) All inequalities in Theorem~\ref{T:rosenthal} are optimal. For assertions $(a)$--$(d)$ it is witnessed by the canonical basis of $\ell_1$ (real or complex). For assertion $(e)$ it is witnessed by Example~\ref{ex:cantor} below.

 (2) Inequality $(d)$ in Theorem~\ref{T:rosenthal} cannot be reversed (even using another constant) for trivial reasons. Indeed, assume $X=\ell_1$ and let $(x_n)$ be the sequence
 $$e_1,0,e_2,0,e_3,0,\dots,$$
 where $(e_n)$ is the canonical basis. Then $\wde{x_n}=0$ and $\dcjr{x_n}=1$.

 (3) Assume $X$ is complex. Then clearly $\cj{x_n}\le \cjr{x_n}$ for each bounded sequence $(x_n)$ in $X$. However, no converse of this inequality is valid, again for trivial reasons. Indeed, let $X=\ell_1$ and $(x_n)$ be the sequence
 $$e_1,ie_1,e_2,ie_2,e_3,ie_3,\dots.$$
 Then $\cj{x_n}=0$ and $\cjr{x_n}=\frac1{\sqrt{2}}$.

 (4) If $X$ is complex, it follows by combining assertions $(e)$ and $(b)$ of Theorem~\ref{T:rosenthal} that $\cjr{x_n}\le \frac\pi2\dcj{x_n}$  for each bounded sequence $(x_n)$ in $X$. It easily follows that
 $$\dcj{x_n}\le \dcjr{x_n}\le\tfrac\pi2\dcj{x_n}.$$
Both inequalities are optimal. For the first one it is witnessed by the canonical basis of the complex space $\ell_1$.  For the second one it is witnessed by the following example.
\end{remarks}

\begin{example}\label{ex:cantor}
    Let $K=\{-1,1\}^\en$ and $X=C(K,\ce)$. Let $f_n\in X$ be defined as the projection onto $n$-th coordinate. Then $\de{f_n}=\wde{f_n}=2$, $\cjr{f_n}=\dcjr{f_n}=1$ and $\dcj{f_n}=\frac2{\pi}$.
\end{example}

\begin{proof}
    First observe, that in the real setting the sequence $(f_n)$ is isometrically equivalent to the canonical basis of $\ell_1$. Indeed, fix $n\in\en$ and $\alpha_1,\dots,\alpha_n\in\er$. Let
    $t\in K$ be such that $t_j=1$ if $\alpha_j\ge0$ and $t_j=-1$ if $\alpha_j<0$ for $j=1,\dots,n$. Then
    $$\sum_{j=1}^n\abs{\alpha_j}\ge\norm{\sum_{j=1}^n\alpha_jf_j}\ge \sum_{j=1}^n\alpha_jf_j(t)=\sum_{j=1}^n\abs{\alpha_j},$$
    so the equalities hold. We easily deduce  $\cjr{f_n}=\dcjr{f_n}=1$ and  $\de{f_n}=\wde{f_n}=2$.

    To prove that $\dcj{f_n}\le\frac2\pi$, fix any $\eta>0$. By Lemma~\ref{L:complex} there is some $m\in\en$ and complex numbers $\alpha_1,\dots,\alpha_m$ such that $\sum_{j=1}^m\abs{\alpha_j}=1$ and for each $I\subset \{1,\dots,m\}$ we have $\abs{\sum_{j\in I}\alpha_j}\le\frac1\pi+\eta$. Let $k_1<k_2<\dots<k_m$ be an arbitrary choice of natural numbers. Then
    $$\begin{aligned}
      \norm{\sum_{j=1}^m \alpha_j f_{k_j}}&=\max\left\{\abs{\sum_{j=1}^m \alpha_j s_j}\setsep s_1,\dots,s_m\in\{-1,1\}\right\}  
      \\  &= \max\left\{\abs{\sum_{j\in I} \alpha_j - \sum_{j\in \{1,\dots,m\}\setminus I} \alpha_j }\setsep I\subset\{1,\dots,m\}\right\}  
      \\&\le  \max\left\{\abs{\sum_{j\in I} \alpha_j} +\abs{\sum_{j\in \{1,\dots,m\}\setminus I} \alpha_j }\setsep I\subset\{1,\dots,m\}\right\}  
      \\&\le\frac2\pi+2\eta.
    \end{aligned}$$
   It easily follows that $\dcj{f_k}\le \frac2\pi+2\eta$. Since $\eta>0$ was arbitrary, the argument is complete.
    
\end{proof}

Now we are ready to formulate the promised quantitative version of the characterization of
weakly precompact sets using non-containment of an $\ell_1$-sequence. It is essentially a reformulation of \cite[Theorem 2.1]{rodr-ep-wpc}.
 
\begin{prop}\label{P:betaw}
    Let $X$ be a Banach space and let $A\subset X$ be a bounded set. Then
    $$\beta_w(A) = 2 \sup\{ \cjr{x_n}\setsep (x_n)\subset A\}
    = 2 \sup\{ \dcjr{x_n}\setsep (x_n)\subset A\}.$$
\end{prop}

\begin{proof}
    The second equality follows easily from the definitions of the respective quantities. 
    The first equality follows by combining assertions $(b)$ and $(d)$ from Theorem~\ref{T:rosenthal} 
    using the second one. 
\end{proof}

It is clear that $\beta_w(A)=0$ if and only if $A$ is weakly precompact, i.e., it may be viewed as a `measure of weak non-precompactness'. Note that any relatively weakly compact set is weakly precompact (due to the Eberlein-\v{S}mulyan theorem). A quantiative version of this implication is contained in the following lemma.

\begin{lemma}\label{L:betaw-wck}
     Let $X$ be a Banach space and let $A\subset X$ be a bounded set. Then $\beta_w(A)\le2\wck{A}$.
\end{lemma}

\begin{proof}
    This is proved in \cite[Theorem 3.2(2)]{kania-qgp}. We recall the easy argument for the sake of completeness.
    Let $(x_n)\subset A$. By \cite[Lemma 5]{wesecom} we know that
    $$\dist(\clu{x_n},X)\ge \cjr{x_n}.$$
    We now conclude using definitions and Proposition~\ref{P:betaw}.
\end{proof}

\section{Quantities in weakly sequentially complete spaces}\label{sec:wsc}

Recall that a Banach space $X$ is \emph{weakly sequentially complete} if any weakly Cauchy sequence is weakly convergent. It follows that in such spaces weakly precompact set coincide with relatively weakly compact sets. In particular, the quantity $\beta_w$ becomes a measure of weak non-compactness.
Hence, in weakly sequentially complete spaces we have three a priori non-equivalent measures of weak non-compactness related by inequalities
$$\beta_w(A)\le 2\wck{A}\le2\omega(A).$$
A possible converse to the first inequality is addressed in the following proposition. We note that condition $(i_c)$ means just that $X$ is $c$-wsc.

\begin{prop}\label{P:wsc}
    Let $X$ be a Banach space. Given $c>0$, consider the following conditions:
    \begin{enumerate}[$(i_c)$]
        \item $\dh(\clu{x_n},X)\le c\de{x_n}$ for each bounded sequence $(x_n)$ in $X$;
        \item $\dist(\clu{x_n},X)\le c\de{x_n}$ for each bounded sequence $(x_n)$ in $X$;
        \item $\dist(\clu{x_n},X)\le c\wde{x_n}$ for each bounded sequence $(x_n)$ in $X$;
        \item $\wck{A}\le c\beta_w(A)$ for each bounded set $A\subset X$.
    \end{enumerate}
   For each $c>0$ we then have 
   $$(i_c)\implies (ii_c)\iff(iii_c)\iff(iv_c)\implies (i_{2c}).$$
\end{prop}

\begin{proof}
    Implications $(i_c)\implies(ii_c)$ and $(iii_c)\implies(iv_c)$ are trivial. 

    $(ii_c)\implies(iii_c)$: Assume $(ii_c)$ holds and let $(x_n)$ be a bounded sequence in $X$. For any subsequence $(x_{n_k})$ we have
    $$\dist(\clu{x_n},X)\le \dist(\clu{x_{n_k}},X)\le c\de{x_{n_k}}.$$
    Indeed, the first inequality follows from the obvious inclusion $\clu{x_{n_k}}\subset\clu{x_n}$ and the second one follows from $(ii_c)$. By passing to the infimum we get
     $$\dist(\clu{x_n},X)\le  c\wde{x_{n}},$$ which completes the argument.

     $(iv_c)\implies (ii_c)$: Assume $(ii_c)$ fails. It means that there is a bounded sequence $(x_n)$ with $\dist(\clu{x_n},X)> c\de{x_n}$. Set $A=\{x_n\setsep n\in \en\}$. Then
     $$\wck{A}\ge \dist(\clu{x_n},X)> c\de{x_n}\ge c\beta_w(A).$$
     The last inequality follows from the fact that given a sequence $(y_k)$ in $A$, we have
     $$\wde{y_k}\le\de{y_k}\le \de{x_n}.$$
     We conclude that $(iv_c)$ fails.

     $(iv_c)\implies (i_{2c})$. Assume $(iv_c)$ holds and let $(x_n)$ be a bounded sequence in $X$. Let $A=\{x_n\setsep n\in\en\}$. Then
     $$\dh(\clu{x_n},X)=\wk{A}\le2\wck{A}\le2c\beta_w(A)\le2c\de{x_n},$$
     where the first inequality follows from \eqref{eq:wck-wk} and the last one follows from the argument in the proof of the previous implication.
\end{proof}

\begin{cor}\label{cor:qwsc}
    Let $X$ be a Banach space. Then quantities $\wck{\cdot}$ and $\beta_w(\cdot)$ are equivalent in $X$ if and only if $X$ enjoys a quantitative version of weak sequential completeness. In particular, these quantities are not equivalent for the space from \cite[Example 4]{wesecom}.
\end{cor}

\begin{cor}\label{cor:Lembedded}
     Let $X$ be an $L$-embedded Banach space. Then $\beta_w(A)=2\wck{A}$ for each bounded set $A\subset X$.
\end{cor}

\begin{proof}
    Inequality `$\le$' holds always by Lemma~\ref{L:betaw-wck}. If $X$ is $L$-embedded, \cite[Theorem 1]{wesecom} shows that $X$ is $\frac12$-wsc, i.e., condition $(i_{\frac12})$ from Proposition~\ref{P:wsc} is fulfilled. Thus Proposition~\ref{P:wsc} yields inequality `$\ge$'.
\end{proof}

\section{Quantitative Schur property}\label{sec:schur}

Recall that a Banach space $X$ has the \emph{Schur property} (or, it is a \emph{Schur space}) if any weakly convergent sequence is norm-convergent. It is known and easy to see that the Schur property implies weak sequential completeness. Thus in Schur spaces $\beta_w$ is a measure of non-compactness. Clearly $\beta_w\le\beta$, a possible converse is addressed in the following proposition. Note that condition $(ii_c)$ means that $X$ has the $c$-Schur property.

\begin{prop}\label{P:schur}
  Let $X$ be a Banach space. Given $c>0$, consider the following conditions:
    \begin{enumerate}[$(i_c)$]
       \item $\limsup\norm{x_n}\le c\cdot\dh(\clu{x_n},\{0\})$ for each bounded sequence $(x_n)$ in $X$;
       \item $\ca{x_n}\le c\de{x_n}$ for each bounded sequence $(x_n)$ in $X$;
       \item $\wca{x_n}\le c\de{x_n}$ for each bounded sequence $(x_n)$ in $X$;
        \item $\wca{x_n}\le c\wde{x_n}$ for each bounded sequence $(x_n)$ in $X$;
        \item $\beta(A)\le c\beta_w(A)$ for each bounded set $A\subset X$.
    \end{enumerate}
   For each $c>0$ we then have 
  $$(i_c)\iff (ii_c)\implies (iii_c)\iff(iv_c)\iff (v_c)\implies (i_{2c+1}).$$
   \iffalse
   $$\begin{array}{ccccccccc}
         (v_c)&\iff&(vi_c)&&&&&& \\
         \Downarrow &&&&&&&& \\
         (i_c)&\implies &(ii_c)&\iff&(iii_c)&\iff&(iv_c)&\implies& (i_{2c+2})  \\
         \Downarrow  &&&&&&&&\\  (v_{2c})  &&&&&&&&
   \end{array}
  .$$\fi
\end{prop}   

\begin{proof}

    $(i_c)\implies(ii_c)$: Assume $(i_c)$ holds and let $(x_n)$ be a bounded sequence in $X$. If $\ca{x_n}=0$, the inequality in $(ii_c)$ clearly holds. Assume $\ca{x_n}>0$ and fix an arbitrary $d\in(0,\ca{x_n})$. Fix $\ep>0$ such that $\ca{x_n}>d+\ep$. Then we may find increasing sequences
    $$l_1<m_1<l_2<m_2<\dots$$
    such that $\norm{x_{l_k}-x_{m_k}}>d+\ep$ for each $k\in\en$. Set $y_k=x_{l_k}-x_{m_k}$. It follows from $(i_c)$ that $\dh(\clu{y_k},\{0\})\ge\frac{d+\ep}{c}$. Thus there is $y^{**}\in\clu{y_k}$ with $\norm{y^{**}}>\frac dc$. Since $y^{**}$ is clearly the difference of some elements of $\clu{x_n}$, we deduce that
    $$\de{x_n}=\diam\clu{x_n}>\tfrac dc.$$
    Since $d$ is arbitrary, the argument is complete.

     $(ii_c)\implies (i_c)$: Assume $(ii_c)$ holds and let $(x_n)$ be a bounded sequence in $X$.  If $\limsup\norm{x_n}=0$, the inequality in $(i_c)$ clearly holds. So, assume that $\limsup \norm{x_n}>0$ and fix and arbitrary $d\in (0,\limsup\norm{x_n})$. %Fix $\ep>0$ such that $\limsup \norm{x_n}>d+\ep$.
     Denote by $(y_n)$ the sequence
    $$x_1,-x_1,x_2,-x_2,x_3,-x_3,\dots$$
    Then $\ca{y_n}>2d$ and hence property $(ii_c)$ yields $\de{y_n}>\frac{2d}{c}$. So, there are $y_1^{**},y_2^{**}\in\clu{y_n}$ with $\norm{y_1^{**}-y_2^{**}}>\frac{2d}{c}$. By the triangle inequality 
    we deduce that $\norm{y_1^{**}}>\frac dc$ or $\norm{y_2^{**}}>\frac dc$. Since weak$^*$-cluster points of $(y_n)$ are clearly weak$^*$-cluster points either of $(x_n)$ or of $(-x_n)$, we deduce 
    $\dh(\clu{x_n},\{0\})>\frac dc$, which completes the argument.

    $(ii_c)\implies(iii_c)$: This is trivial as $\wca{x_n}\le \ca{x_n}$.
  
    $(iii_c)\implies(iv_c)$: Assume $(iii_c)$ holds and let $(x_n)$ be a bounded sequence in $X$. Let $(x_{n_k})$ be any subsequence of $(x_n)$. Then
    $$\wca{x_n}\le\wca{x_{n_k}}\le\de{x_{n_k}}.$$
    By passing to the infimum we deduce $\wca{x_n}\le c\wde{x_n}$.
    
    $(iv_c)\implies(v_c)$: This follows immediately from the definitions.

    $(v_c)\implies(iii_c)$: Assume $(v_c)$ holds and let $(x_n)$ be a bounded sequence in $X$.
    Let $A=\{x_{n}\setsep n\in\en\}$. Using definitions and validity of $(v_c)$ we get
    $$\wca{x_n}\le  \beta(A)\le c\beta_w(A)\le c\de{x_{n}},$$
    which completes the argument.

    $(iii_c)\implies(i_{2c+1})$: Assume that $(iii_c)$ holds. Then $X$ has the Schur property. Indeed, if $(x_n)$ is weakly convergent, then $\de{x_n}=0$. It follows that $\wca{x_n}=0$. The same holds for any subsequence of $(x_n)$, hence $\beta(\{x_n\setsep n\in\en\})=0$. I.e., $\{x_n\setsep n\in\en\}$ is a relatively compact set. Since $(x_n)$ is weakly convergent, it easily follows that it is also norm convergent. This proves the Schur property.

    Now we proceed by contradiction. Assume $(i_{2c+1})$ fails. Then there is a bounded sequence $(x_n)$ such that $\limsup\norm{x_n}>(2c+1)\dh(\clu{x_n},\{0\})$. Since $X$ has the Schur property, necessarily $\dh(\clu{x_n},\{0\})>0$. Up to multiplying by a positive constant we may assume that $\dh(\clu{x_n},\{0\})<1$ and $\limsup\norm{x_n}>2c+1$. Fix $d<1$ such that $\dh(\clu{x_n},\{0\})<d$. Up to passing to a subsequence we may assume that $\norm{x_{n}}>2c+1$ for each $n\in\en$. 
    By induction we construct an increasing sequence $(n_k)$ of positive integers and a sequence $(x_k^*)$ of norm-one elements of $X^*$ such that the following conditions are fulfilled.
    \begin{enumerate}[$(a)$]
        \item $n_1=1$;
        \item $\abs{\ip{x_k^*}{x_{n_{k}}}}>2c+1$;
        \item $\abs{x_k^*(x_j)}<d$ for $j\ge n_{k+1}$.
    \end{enumerate}

    Indeed, we may surely set $n_1=1$. Assume that $k\in\en$ and that we have already constructed $n_k$.
    Since $\norm{x_{n_k}}>2c+1$, we may find a norm-one element $x_k^*\in X^*$ such that $(b)$ is satisfied. Further, 
    $$\limsup_n\abs{x_k^*(x_n)}\le\dh(\clu{x_n},\{0\})<d,$$ 
    so we may find $n_{k+1}$ fulfilling $(c)$. This completes the construction.
    
    Fix $k,l\in\en$ such that $k<l$. Then
    $$\norm{x_{n_k}-x_{n_l}}\ge\abs{\ip{x_k^*}{x_{n_k}}-\ip{x_k^*}{x_{n_l}}}>2c+1-d>2c.$$
    Thus $\wca{x_{n_k}}\ge 2c$
    On the other hand,
$$\clu{x_{n_k}}\subset \clu{x_n}\subset U(0,d),$$
thus $\de{x_{n_k}}\le 2d<2$. This contradicts $(iii_c)$.                
\end{proof}

\begin{remark} 
    Example~\ref{ex:LF} below provides a space with satisfies $(iii_1)$, $(i_2)$ but not $(i_c)$ for any $c<2$.
    It follows that some increase of the constant $c$ in implication $(iii)\implies (i)$ is necessary.
    However, it is not clear whether the result above is optimal.
\end{remark}

\begin{cor}\label{cor:qschur}
    Let $X$ be a Banach space. The quantities
    $$\beta_w(\cdot),\wck{\cdot},\omega(\cdot),\chi(\cdot)$$
    are equivalent if and only if $X$ enjoys the quantitative Schur property.
\end{cor}

\begin{proof}
    In general we have
    $$\beta_w(A)\le 2\wck{A}\le2\omega(A)\le2\chi(A)\le2\beta(A)\le 4\chi(A).$$
    By Proposition~\ref{P:schur} we see that $\beta_w$ is equivalent to $\beta$ if and only if $X$ has the quantitative Schur property. This completes the argument.
\end{proof}

Let us remark that the previous corollary is a bit extended version of \cite[Theorem 2.1(ii)]{qschur-dp}.
By combining Corollary~\ref{cor:qschur} with Corollary~\ref{cor:qwsc} we get the following corollary:

\begin{cor}
   Let $X$ be a Banach space which is quantitatively weakly sequentially complete and has the Schur property. Then quantities $\wk{\cdot}$ and $\omega(\cdot)$ are equivalent in $X$ if and only if $X$ has the quantitative Schur property. 
\end{cor}

Some applications of this corollary are collected in the following example.

\begin{example2}
(1) Let $X$ be the space from \cite[Example 1.4]{qschur}. Then $X$ is $L$-embedded and thus $\frac12$-wsc. Further, it has the Schur property, but not its quantitative version. Hence, quantities $\wk{\cdot}$ and $\omega(\cdot)$ are not equivalent in $X$.

(2) Let $X$ be the space from  \cite[Example 10.1]{qdpp}. Then $X^*$ is $L$-embedded and has the Schur property. Moreover, quantities $\wk{\cdot}$ and $\omega(\cdot)$ are not equivalent in $X^*$. Thus $X^*$ fails the quantitative Schur property. (This also follows from \cite[Theorem 2.1]{qschur-dp}.)
 
(3) It follows by combining \cite[Proposition 17]{petitjean-schur} with \cite[Theorem 3.2]{p1u} that the Lipschitz-free space over a proper purely $1$-unrectifiable metric space has the $1$-Schur property. Therefore, in such spaces quantities $\wk{\cdot}$ and $\omega(\cdot)$ are equivalent.
\end{example2}

In the literature another quantification of the Schur property appears as well.  In \cite[p. 57]{Go-Ka-Li} a Banach space $X$ is said to have \emph{the $1$-strong Schur property} if for any $\delta\in(0,2]$, any $\varepsilon>0$ and any normalized $\delta$-separated sequence in $X$ there is a subsequence which is $(\frac2\delta+\varepsilon)$-equivalent to the standard basis of $\ell_1$. This notion is used also in \cite{p1u}. In \cite[Proposition 4.1]{qschur} is proved that, for real Banach spaces, the $1$-Schur property implies the $1$-strong Schur property and that the $1$-strong Schur property implies the $5$-Schur property. We now provide a more precise version of this relationship.

\begin{prop}\label{P:strongschur}
Let $X$ be Banach space and let $c\ge 1$. Consider property $(iv_c)$ from Proposition~\ref{P:schur} ant the following property:
\begin{enumerate}[$(vi_c)$]
    \item For any $\delta\in(0,2]$, any $\varepsilon>0$ and any normalized $\delta$-separated sequence in $X$ there is a subsequence which is $(\frac{2c}\delta+\varepsilon)$-equivalent to the standard basis of $\ell_1$.
\end{enumerate}   
Then
\begin{itemize}
    \item $(iv_c)\iff(vi_c)$ if $X$ is a real space;
    \item $(vi_c)\implies (iv_c)\implies(vi_{\pi c/2})$ if $X$ is a complex space.
\end{itemize}
\end{prop}

\begin{proof}  The proof is done by a minor but careful modification of the proof of \cite[Proposition 4.1]{qschur}:

$(iv_c)\implies(vi_c)$ if $X$ is a real space: Fix $\delta\in(0,2]$ and $\varepsilon>0$. Let $(x_n)$ be a normalized $\delta$-separated sequence in $X$. Then clearly $\wca{x_n}\ge\delta$. Property $(iv_c)$ yields $\wde{x_n}\ge\frac\delta c$. We use Theorem~\ref{T:rosenthal}$(d)$ to deduce $\dcj{x_k}\ge \frac{\delta}{2c}$. Now it easily follows that there is a subsequence $(\frac{2c}\delta+\varepsilon)$-equivalent to the canonical basis of $\ell_1$.

$(iv_c)\implies (vi_{\pi c/2})$ if $X$ is a complex space: We proceed similarly, we only use
 Theorem~\ref{T:rosenthal}$(e)$ to deduce $\dcj{x_k}\ge \frac{\delta}{\pi c}$. Now it easily follows that there is a subsequence $(\frac{\pi c}\delta+\varepsilon)$-equivalent to the canonical basis of $\ell_1$.

$(vi_c)\implies (iv_c)$: Let $(x_n)$ be a bounded sequence in $X$. If $\wca{x_n}=0$, the inequality from $(iv_c)$ is trivial. So, assume that $\wca{x_n}>0$ and fix an arbitrary $d\in (0,\wca{x_n})$. Further, let $\varepsilon>0$ be arbitrary. Then $\wca{x_n}>d$, so, an easy application of the classical Ramsey theorem shows that there is a $d$-separated subsequence $(x_{n_k})$.
Up to passing to a further subsequence, we may assume that $\norm{x_{n_k}}\to \alpha \ge0$. Since $(x_{n_k})$ is $d$-separated, necessarily $\alpha\ge \frac d2>0$. Up to omitting finitely many elements we may assume that $\abs{\alpha-\norm{x_{n_k}}}<\varepsilon$  for each $k\in\en$. Then the sequence $\left(\frac{x_{n_k}}{\norm{x_{n_k}}}\right)$ is normalized and $(\frac d\alpha-2\varepsilon)$-separated (see the computation in the proof of \cite[Proposition 4.1]{qschur}). By $(vi_c)$ we may assume, up to passing to a subsequence that
$$\left(\frac{x_{n_k}}{\norm{x_{n_k}}}\right) \mbox{ is }\left(\frac {2c}{\frac d\alpha-2\varepsilon}+\varepsilon\right)\mbox{-equivalent to the standard basis of }\ell_1.$$ 
Then
$$\cj{\frac{x_{n_k}}{\norm{x_{n_k}}}}\ge\frac{1}{\frac {2c}{\frac d\alpha-2\varepsilon}+\varepsilon},$$
hence
$$\dcj{x_n}\ge \cj{x_{n_k}}\ge\frac{\alpha}{\frac {2c}{\frac d\alpha-2\varepsilon}+\varepsilon}.$$
Since $\varepsilon>0$ is arbitrary, we deduce that
$$\dcj{x_n}\ge \frac{\alpha}{\frac {2c}{\frac d\alpha}}=\frac d{2c}.$$
By Theorem~\ref{T:rosenthal}$(c)$ we deduce
$$\de{x_n}\ge \frac dc.$$
Since $d$ was arbitrary, we get $\wca{x_n}\le c\de{x_n}$; i.e., condition $(iii_c)$ of Proposition~\ref{P:schur} is valid. We conclude by recalling that $(iii_c)\iff (iv_c)$ by Proposition~\ref{P:schur}.
\end{proof}

\begin{remark}
 The previous proposition together with Proposition~\ref{P:schur} provides an improvement of \cite[Proposition 4.1]{qschur}. Indeed, if $X$ has the $1$-strong Schur property, it satisfies property $(vi_1)$ and hence property $(iv_1)$. Proposition~\ref{P:schur} then yields validity of $(i_3)$, i.e., the $3$-Schur property.
\end{remark}

Next we look at complex spaces. We illustrate by an example that the properties of complex spaces are indeed different. The first ingredient is the following lemma on complexification of Schur spaces.
    
\begin{lemma}\label{L:komplexifikace}
    Let $X$ be a real Banach space and let $X_C$ denote its complexification. Assume that $c>0$.
    \begin{enumerate}[$(a)$]
        \item $X$ has the $c$-Schur property if and only if $X_C$ has the $c$-Schur property.
        \item $X$ satisfies condition $(iii_c)$ from Proposition~\ref{P:schur} if only if this condition is satisfied by $X_C$.
    \end{enumerate}
\end{lemma}

\begin{proof}
  Let us recall that the complexification of $X$ (see, e.g. \cite[p. 54]{fhhmpz}) is the complex vector space
  $$X_C=X+iX=\{x+iy\setsep x,y\in X\}$$
  equipped with the norm
  $$\norm{x+iy}=\sup\{\norm{\alpha x+\beta y}\setsep \alpha,\beta\in\er,\alpha^2+\beta^2=1\}.$$
  Let us first prove the `if' parts of both assertions. Denote $(X_C)_R$ the real version of the complex
  space $X_C$. If $X_C$ has one of the properties, then $(X_C)_R$ has it clearly as well. Further, Since $X$
 is isometric to a subspace of $(X_C)_R$ and both properties are inherited by subspaces, the `if' parts follow.  
  We continue by proving the `only if' parts.
  
  $(a)$: Assume that $X$ has the $c$-Schur property, i.e., it satisfies condition $(i_c)$ from Proposition~\ref{P:schur} and let us prove that $X_C$ satisfies it as well.
  Let $(z_n)=(x_n+iy_n)$ be a bounded sequence in $X_C$ such that $\limsup\norm{z_n}>d>0$. Fix $\ep>0$ such that $\limsup\norm{z_n}>d+\ep$. Up to passing to a subsequence we may assume that $\norm{z_n}>d+\ep$ for each $n\in\en$.
  By definition of the norm on $X_C$ there are $(\alpha_n,\beta_n)\in\er^2$ with $\alpha_n^2+\beta_n^2=1$ such that
  $$\norm{\alpha_n x_n+\beta_n y_n}>d+\ep\mbox{ for each }n\in\en.$$
  Up to passing to a subsequence we may assume that $(\alpha_n,\beta_n)\to(\alpha,\beta)\in\er^2$. Then clearly $\alpha^2+\beta^2=1$ and $\limsup \norm{\alpha x_n+\beta y_n}\ge d+\ep$. We apply validity of $(i_c)$ for $X$ to find $x^{**}\in \clu{\alpha x_n+\beta x_n}$ with $\norm{x^{**}}>\frac{d}{c}$. Then there is $x^{*}\in X^*$ of norm one such that $\ip{x^{**}}{x^*}>\frac dc$. Thus
  $$\ip{x^*}{\alpha x_n+\beta x_n}>\tfrac{d}{c}\mbox{ for infinitely many }n\in\en.$$
  Let us define $\widetilde{x^*}\in (X_C)^*$ by setting
  $$\widetilde{x^*}(x+iy)=x^*(x)+i x^*(y),\quad x+iy\in X_C.$$
  It is clearly a linear functional on $X_C$. Moreover, $\norm{\widetilde{x^*}}\le 1$. Indeed, assume that $x+iy\in X_C$. Then
  $$\begin{aligned}
   \abs{\widetilde{x^*}(x+iy)}&=\abs{x^*(x)+i x^*(y)} = \sup\{ \abs{\gamma x^*(x)+\delta x^*(y)}\setsep \gamma,\delta\in\er,\gamma^2+\delta^2=1\}\\
   &
   =\sup\{ \abs{x^*(\gamma x+\delta y)}\setsep \gamma,\delta\in\er,\gamma^2+\delta^2=1\}\\&\le 
   \sup\{ \norm{\gamma x+\delta y}\setsep \gamma,\delta\in\er,\gamma^2+\delta^2=1\}=\norm{x+iy}.
  \end{aligned}$$
Moreover, for infinitely many $n\in\en$ we have
$$\abs{\widetilde{x^*}(x_n+iy_n)}=\abs{x^*(x_n)+i x^*(y_n)}\ge \alpha x^*(x_n)+\beta x^*(y_n)=x^*(\alpha x_n+\beta y_n)>\tfrac dc.$$
Now it is clear there there is a weak$^*$-cluster point of $(x_n+i y_n)$ in $(X_C)^{**}$ of norm at least $\frac dc$. This completes the proof.  

$(b)$: Assume that $X$ satisfies condition $(iii_c)$ from Proposition~\ref{P:schur} and let us prove that $X_C$ satisfies it as well.  Let $(z_n)=(x_n+iy_n)$ be a bounded sequence in $X_C$ such that $\wca{z_n}>d>0$. Let $M>0$ be such that $\norm{x_n}\le M$ and $\norm{y_n}\le M$ for each $\ep>0$. Fix $\ep>0$ such that $\wca{z_n}>d+2\ep$. Due to the classical Ramsey theorem, we may assume, up to passing to a subsequence that $\norm{z_n-z_m}>d+2\ep$ whenever $n\ne m$. Hence, for each pair $n\ne m$ there is some vector $(\alpha_{n,m},\beta_{n,m})\in\er^2$ with $\alpha_{n,m}^2+\beta_{n,m}^2=1$ such that
$\norm{\alpha_{n,m} (x_n-x_m)+\beta_{n,m}(y_n-y_m)}>d+2\ep$. 

Fix $\{(\gamma_j,\delta_j)\setsep j=1,\dots,N\}$ a (finite) subset of the unit circle in $\er^2$ which is an $\frac{\ep}{2M}$-net in the $\ell^1$-norm. By the classical Ramsey theorem we may assume, up to passing to a subsequence that there is some $j\in\{1,\dots,N\}$ such that
$$\norm{(\alpha_{n,m},\beta_{n,m})-(\gamma_j,\delta_j)}_1<\tfrac{\ep}{2M} \mbox{ whenever }n\ne m.$$
Set $\gamma=\gamma_j$ and $\delta=\delta_j$. Then for each $n\ne m$ we have
$$\begin{aligned}
 \norm{\gamma(x_n-x_m)+\delta(y_n-y_m)}&\ge  \norm{\alpha_{n,m} (x_n-x_m)+\beta_{n,m}(y_n-y_m)}\\&\qquad-\norm{(\gamma-\alpha_{n,m}) (x_n-x_m)+(\delta-\beta_{n,m})(y_n-y_m)}\\&>
 d+2\ep-\tfrac{\ep}{2M}\cdot 2M=d+\ep.
\end{aligned}$$
It follows that $\wca{\gamma x_n+\delta y_n}\ge d+\ep$. By condition $(iii_c)$ in $X$ we deduce that $\de{\gamma x_n+\delta y_n}\ge \frac{d+\ep}{c}$. Hence, there is a norm-one element $x^*\in X^*$ such that
$\ca{\ip{x^*}{\gamma x_n+\delta y_n}}>\frac dc$. Define $\widetilde{x^*}\in (X_C)^*$ as in the proof of $(a)$ and recall it has norm one. For $m,n\in\en$ we have
$$\begin{aligned}
\abs{\widetilde{x^*}(z_n-z_m)}&=\abs{x^*(x_n-x_m)+ix^*(y_n-y_m)}
\ge \abs{\gamma x^*(x_m-x_m)+\delta x^*(y_n-y_m)}
\\&=\abs{x^*(\gamma x_n+\delta y_n)-x^*(\gamma x_m+\delta y_m)}.
\end{aligned}$$
It follows that 
$$\de{z_n}\ge \ca{\widetilde{x^*}(z_n)}\ge \ca{x^*(\gamma x_n+\delta y_n)}>\tfrac dc.$$
This completes the proof.

\end{proof}

We continue by an example showing optimality of the inequalities in the complex case of Proposition~\ref{P:strongschur}.

\begin{example}\label{ex:ell_1-complexification}
     \begin{enumerate}[$(1)$]
       \item The complex Banach space $\ell_1$ has both the $1$-Schur property and the $1$-strong Schur property.
      \item Let $X$ denote the complexification of the real Banach space $\ell_1$. Then:
    \begin{enumerate}[$(a)$]
        \item $X$ has the $1$-Schur property.
        \item $X$ fails property $(vi_c)$ from Propostion~\ref{P:strongschur} for any $c<\frac\pi2$. In particular, $X$ fails the $1$-strong Schur property.
    \end{enumerate}
       \end{enumerate}
\end{example}

\begin{proof}
     $(1)$:  The space $\ell_1$ has the $1$-Schur property by \cite[Theorem 1.3]{qschur}. (The quoted theorem deals with the real case, but its proof works in the complex case as well, as explained in the final remarks of \cite{qschur}.)  A proof of the $1$-strong Schur property may be done by a modification of the proof of \cite[Theorem 1.3]{qschur}. Let us indicate it.

     Let $\delta\in(0,2]$, let $(x_n)$ be a normalized $\delta$-separated sequence in $\ell_1$ and let $\ep>0$. Without loss of generality assume $\ep<\frac\delta4$. Up to passing to a subsequence we may assume that $(x_n)$ pointwise converges to some $x\in\ell_1$. Set $y_n=x_n-x$. Then $(y_n)$ is $\delta$-separated and pointwise converges to $0$. Up to omitting one element, we may assume that $\norm{y_n}\ge\frac{\delta}{2}$ for each $n\in\en$. By an easy induction we find a subsequence $(y_{n_k})$
     and integers 
     $$0=N_0<N_1<N_2<\dots$$
     such that $\norm{y_{n_k}|_{(N_{k-1},N_k]}}>\norm{y_{n_k}}-\frac{\ep}{2^k}$ for each $k\in\en$.
     Assume that $m\in\en$ and $\alpha_1,\dots,\alpha_m\in\ce$. Then
     $$\begin{aligned}
        \norm{\sum_{j=1}^m\alpha_j y_{n_j}}&\ge \sum_{k=1}^m\norm{\sum_{j=1}^m\alpha_j y_{n_j}|_{(N_{k-1},N_k]}}
        \\&\ge \sum_{k=1}^m\left(\abs{\alpha_k}\norm{y_{n_k}|_{(N_{k-1},N_k]}}-\sum_{1\le j\le m, j\ne k}\abs{\alpha_j}\norm{ y_{n_j}|_{(N_{k-1},N_k]}}\right)
        \\&=  \sum_{k=1}^m\abs{\alpha_k}(\norm{y_{n_k}|_{(N_{k-1},N_k]}} - \norm{y_{n_k}|_{(0,N_m]\setminus (N_{k-1},N_k]}})\\&\ge(\tfrac\delta2-2\ep)\sum_{j=1}^m \abs{\alpha_j}.
     \end{aligned}$$
    It means that $(y_{n_k})$ satisfies lower $\ell_1$-estimate with constant $\frac\delta2-2\ep$. By \cite[Lemma 4.2]{knaust-odell} we deduce that, after omitting finitely many elements, the sequence $(x_{n_k})=(y_{n_k}+x)$ satisfies the same lower $\ell_1$-estimate. (Note that the quoted lemma is formulated in the real setting, but the same proof works in the complex case as well). But this means that $(x_{n_k})$ is
    $$\frac{1}{\tfrac\delta2-2\ep}=\frac{2}{\delta-4\ep}\mbox{-equivalent to the $\ell_1$-basis.}$$ This completes the argument.

    $(2)$: The real space $\ell_1$ has the $1$-Schur property by \cite[Theorem 1.3]{qschur}. Hence,  assertion $(a)$ follows from Lemma~\ref{L:komplexifikace}. To prove assertion $(b)$ we first observe that $X$ is canonically isometrically embedded into $C(\{-1,1\}^\en,\ce)$. For $x,y\in \ell_1$ we have
    $$\begin{aligned}
        \norm{x+iy}&=\sup\{\norm{\alpha x+\beta y}\setsep\alpha,\beta\in\er,\alpha^2+\beta^2=1\}
        \\&= \sup\{\sup\{\abs{\ip{x^*}{\alpha x+\beta y}}\setsep x^*\in B_{\ell_\infty}\}\setsep\alpha,\beta\in\er,\alpha^2+\beta^2=1\}
        \\&= \sup\{\sup\{\abs{\ip{x^*}{\alpha x+\beta y}}\setsep x^*\in\ext B_{\ell_\infty}\}\setsep\alpha,\beta\in\er,\alpha^2+\beta^2=1\}
         \\&= \sup\{\sup\{\abs{\alpha x^*(x)+\beta x^*(y)}\setsep\alpha,\beta\in\er,\alpha^2+\beta^2=1 \}\setsep x^*\in\ext B_{\ell_\infty}\}
         \\&=\sup\{\abs{x^*(x)+ix^*(y)}\setsep x^*\in\ext B_{\ell_\infty}\}.
    \end{aligned}$$
   Observe that $\ext B_{\ell_\infty}=\{-1,1\}^\en$ and by the above computation the assigment
   $$ z=x+iy\in X\mapsto f_z, \mbox{ where }f_z(x^*)=x^*(x)+ix^*(y), x^*\in \ext B_{\ell_\infty}=\{-1,1\}^\en,$$
   is a linear isometric inclusion of $X$ into $C(\{-1,1\}^\en,\ce)$. 

    Take the sequence $(e_k)$ of canonical vectors in $X$. This is a normalized $2$-separated sequence. However,
    $f_{e_k}$ is the projection of $\{-1,1\}^\en$ onto $k$-th coordinate. So, by Example~\ref{ex:cantor} we deduce that $\dcj{e_k}=\frac2\pi$. In particular, no subsequence of $(e_k)$ is $c$-equivalent to the $\ell_1$-basis for $c<\frac\pi2$. This completes the proof.
  \end{proof}

\section{A sufficient condition for the $1$-Schur property}\label{sec:m1}

There are several examples of spaces with the $1$-Schur property -- $\ell_1(\Gamma)$ \cite[Theorem 1.3]{qschur},
$X^*$ for any $X\subset c_0(\Gamma)$ \cite[Theorem 1.1]{qschur-dp}, Lipschitz-free spaces over proper purely $1$-unrectifiable metric spaces  (by \cite[Proposition 17]{petitjean-schur} and \cite[Theorem 3.2]{p1u}). A key role in 
the result of \cite{qschur-dp} was played by property $(m_1^*)$ (see \cite[Lemma 1.2]{qschur-dp}). In the present section we provide a generalization of the mentioned result of \cite{qschur-dp} by showing that a variant of property $(m_1^*)$ is a sufficient condition for the $1$-Schur property. We first recall definitions of the relevant properties (cf. \cite[Definition 2.1]{kalton-werner}). 

A Banach space $X$ is said to have
\begin{itemize}
    \item \emph{property $(m_1)$} if $\limsup\norm{x_n+x}=\norm{x}+\limsup\norm{x_n}$  whenever $x\in X$ and $(x_n)$ is a weakly null sequence in $X$;
    \item \emph{property $(m_1^*)$} if $\limsup\norm{x_n^*+x^*}=\norm{x^*}+\limsup\norm{x_n^*}$  whenever $x^*\in X^*$ and $(x_n^*)$ is a weak$^*$ null sequence in $X^*$.
\end{itemize}

We continue by an auxiliary lemma on constructing of an $\ell_1$-sequence from a variant of these properties.

\begin{lemma}\label{L:m1-ell1}
    Let $X$ be a Banach space. Let $(x_n)$ be a bounded sequence in $X$, $x\in X\setminus\{0\}$ and $c>0$. Assume that
    $$\forall y\in \span(\{x\}\cup\{x_n\setsep n\in\en\})\colon \lim_n\norm{x_n+y}=c+\norm{y}.$$
    Then for each $\ep>0$ there is a subsequence $(x_{n_k})$ such that the sequence
    $$\tfrac{x}{\norm{x}}, \tfrac{x_{n_1}}{c}, \tfrac{x_{n_2}}{c}, \tfrac{x_{n_3}}{c},\dots$$
    is $\ep$-isometrically equivalent to the canonical basis of $\ell^1$.
\end{lemma}

\begin{proof}
    It is clear that without loss of generality we may assume that $c=1$, $\norm{x}=1$ and $\ep\in(0,1)$. Then $\norm{x_n}\to1$ (we may apply the assumption to $y=0$). Hence, up to omitting finitely many elements we may assume that $\norm{x_n}<1+\ep$ for each $n\in\en$. So, it is enough to find a subsequence for which the lower $\ell_1$-estimate holds with factor $1-\ep$.

   We will construct by induction an increasing sequence $(n_k)$ of positive integers such that
   $$\norm{\alpha x+\sum_{j=1}^k\alpha_j x_{n_j}}>1-\ep+\tfrac{\ep}{2^k} \mbox{ whenever }\alpha,\alpha_1,\dots,\alpha_k\in\ef, \abs{\alpha}+\sum_{j=1}^k\abs{\alpha_j}=1.$$
   It is clear that the proof will be complete as soon as the construction is performed.

   We start by finding $n_1$. To this end fix $\{(\alpha^l,\alpha_1^l)\setsep l=1,\dots,N\}$ an
   $\frac{\ep}{8}$-net of the unit sphere of $\ell_1^2$. By the assumption we have, for each $l\in\{1,\dots,N\}$,
   $$\lim_n\norm{\alpha^l x+\alpha_1^l x_n}=\abs{\alpha^l}+\abs{\alpha_1^l}=1,$$
   so we may find $n_1\in\en$ such that
   $$\forall n\ge n_1\;\forall l\in\{1,\dots,N\}\colon \norm{\alpha^l x+\alpha_1^l x_n}>1-\tfrac{\ep}{4}.$$
   Let now $(\alpha,\alpha_1)$ be any element of the unit sphere of $\ell_1^2$. Let $l\in\{1,\dots,N\}$
   be such that $\norm{(\alpha,\alpha_1)-(\alpha^l,\alpha_1^l)}<\frac{\ep}{8}$. Then
   $$\begin{aligned}
   \norm{\alpha x+\alpha_1 x_{n_1}}&\ge \norm{\alpha^l x+\alpha_1^l x_{n_1}}-\norm{(\alpha-\alpha^l) x+(\alpha_1 -\alpha_1^l)x_{n_1}} \\&>1-\tfrac{\ep}{4}-(1+\ep)\tfrac{\ep}{8}>1-\tfrac{\ep}{2}.\end{aligned}$$
  This completes the proof of the first step.

   The induction step is analogous: Assume that $n_1,\dots,n_k$ are constructed such that the condition is fulfilled.   Fix $\{(\alpha^l,\alpha_1^l,\dots,\alpha_{k+1}^l)\setsep l=1,\dots,N\}$ an
   $\frac{\ep}{2^{k+3}}$-net of the unit sphere of $\ell_1^{k+2}$. By the assumption together with the induction hypothesis we have, for each $l\in\{1,\dots,N\}$,
   $$\begin{aligned}
       \lim_n\norm{\alpha^l x+\sum_{j=1}^k\alpha_j^l x_{n_j}+\alpha_{k+1}^l x_n}&=\norm{\alpha^l x+\sum_{j=1}^k\alpha_j^l x_{n_j}}+\abs{\alpha_{k+1}^l}\\
       &\ge (1-\ep+\tfrac{\ep}{2^k})\sum_{j=1}^k\abs{\alpha_j^l}+\abs{\alpha_{k+1}^l}\ge 1-\ep+\tfrac{\ep}{2^k}
      \end{aligned}$$
  Hence, we may find $n_{k+1}>n_k$ such that     
   $$\forall n\ge n_{k+1}\;\forall l\in\{1,\dots,N\}\colon \norm{\alpha^l x+\sum_{j=1}^k\alpha_j^l x_{n_j}+\alpha_{k+1}^l x_n}>1-\ep+\tfrac{3\ep}{2^{k+2}}.$$

 Let now $(\alpha,\alpha_1,\dots,\alpha_{k+1})$ be any element of the unit sphere of $\ell_1^{k+2}$. Let $l\in\{1,\dots,N\}$
   be such that $\norm{(\alpha,\alpha_1,\dots,\alpha_{k+1})-(\alpha^l,\alpha_1^l,\dots,\alpha_{k+1}^l)}<\frac{\ep}{2^{k+3}}$. Then
   $$\begin{aligned}
   \norm{\alpha x+\sum_{j=1}^{k+1}\alpha_j x_{n_j}}&\ge \norm{\alpha^l x+\sum_{j=1}^{k+1}\alpha_j^l x_{n_j}}-\norm{(\alpha-\alpha^l) x+\sum_{j=1}^{k+1}(\alpha_j -\alpha_j^l)x_{n_j}} \\&>1-\ep+\tfrac{3\ep}{2^{k+2}}-(1+\ep)\tfrac{\ep}{2^{k+3}}>1-\tfrac{\ep}{2^{k+1}}.\end{aligned}$$
   This completes the construction and hence also the proof.\end{proof}

As a corollary we get the following characterization of the Schur property.

\begin{prop}
    Let $X$ be a Banach space. Then $X$ has the Schur property if and only if it satisfies property $(m_1)$.
\end{prop}

\begin{proof}
    Assume that $X$ has the Schur property. Let $(x_n)$ be any weakly null sequence. Then $\norm{x_n}\to0$ and $\norm{x_n+x}\to\norm{x}$ for each $x\in X$. Thus property $(m_1)$ follows.

    Conversely, assume that $X$ satisfies property $(m_1)$ but fails the Schur property. Then there is a weakly null sequence $(x_n)$ in the unit sphere of $X$. It follows that $(x_n)$ satisfies the assumption of Lemma~\ref{L:m1-ell1} with $c=1$. By the quoted lemma we deduce that $(x_n)$ has a subsequence equivalent to the canonical basis of $\ell_1$. So it cannot be weakly null, which is a contradiction.
\end{proof}

We continue by a lemma on weak$^*$ cluster points of $\ell_1$-sequences.

\begin{lemma}\label{L:hrombod}
    Let $X$ be a Banach space and assume that $(x_n)$ is a bounded sequence in $X$ which satisfies lower $\ell_1$-estimates with some $c>0$, i.e.,
    $$ \norm{\sum_{j=1}^n\alpha_jx_j}\ge c\sum_{j=1}^n\abs{\alpha_j}$$
   whenever $n\in\en$ and $\alpha_1,\dots,\alpha_n\in\ef$. Let $a,b\in\ef$. Then any weak$^*$-cluster point of sequence $(ax_1+bx_n)_n$ in $X^{**}$ has norm at least $c(\abs{a}+\abs{b})$.
\end{lemma}

\begin{proof} Clearly, we may assume without loss of generality that $X=\ov{\span\{x_n\setsep n\in\en\}}$. Define $T:\ell_1\to X$ by
$$T((\alpha_n))=\sum_{n=1}^\infty \alpha_n x_n,\quad (\alpha_n)\in \ell_1.$$
By the assumptions we know that $T$ is a surjective isomorphism and $\norm{T^{-1}}\le \frac{1}{c}$. Note that the bidual operator $T^{**}$ has the same properties and it is, moreover, a weak$^*$-to-weak$^*$-homeomorphism. So, each weak$^*$-cluster point of $(ax_1+bx_n)_n$ is of the form
$T^{**}(x^{**})$ where $x^{**}$ is a weak$^*$-cluster point of $(ae_1+be_n)_n$ in $\ell_1^{**}$.
But any such $x^{**}$ may be expressed as $ae_1+bz^{**}$ where $z^{**}$ is a weak$^*$-cluster point of the basis $(e_n)$. Hence $\norm{x^{**}}=\abs{a}+\abs{b}$ and so $\norm{T^{**}(x^{**})}\ge c(\abs{a}+\abs{b})$.
This completes the proof.  
\end{proof}

Now we are ready to formulate and prove the promised sufficient condition for $1$-Schur property.
 
\begin{prop}\label{P:subs-1Schur}
    Let $X$ be a Banach space. Assume that for each bounded sequence $(x_n)$ in $X$ there is a subsequence $(x_{n_k})$ and $x\in X$ such that for any further subsequence $(x_{n_{k_l}})$ and any $y\in \span(\{x\}\cup\{x_{n_k}\setsep k\in\en\}$ we have
    $$\limsup_{l}\norm{x_{n_{k_l}}+y}=\norm{y+x}+\limsup_l
    \norm{x_{n_{k_l}}-x}.$$
    Then $X$ has both the $1$-Schur property and the $1$-strong Schur property.
\end{prop}

\begin{proof} To prove the $1$-Schur property we are going to verify property $(i_1)$ from Proposition~\ref{P:schur}.
    Let $(x_n)$ be a bounded sequence in $X$. Assume that $\limsup_n \norm{x_n}>c>0$. Up to passing to a subsequence, we may assume that $\inf_n \norm{x_n}>c$.
        By the assumption we may suppose, up to passing to a subsequence, that there is some $x\in X$
    such that
        $$\limsup_k\norm{x_{n_{k}}+y}=\norm{x+y}+\limsup_k\norm{x_{n_k}-x}$$
    whenever $y\in \span(\{x\}\cup\{x_{n}\setsep n\in\en\})$ and $(x_{n_k})$ is a subsequence of $(x_n)$. Let us fix a subsequence
    $(x_{n_k})$ such that $\norm{x_{n_k}-x}\to d\in[0,\infty)$. Let us distinguish two cases:

    Case 1: Assume that $d=0$. This means that $x_{n_k}\to x$ in the norm. Hence $x$ is a weak$^*$-cluster point of $(x_n)$ satisfying $\norm{x}\ge c$. Thus $\dh(\clu{x_n},\{0\})\ge c$.

   Case 2: Assume that $d>0$. Then
   $$\limsup_{l}\norm{x_{n_{k_l}}+y}=\norm{y+x}+d$$
   for each subsequence $(x_{n_{k_l}})$ of $(x_{n_k})$ and each $y\in  \span(\{x\}\cup\{x_{n}\setsep n\in\en\})$.
   This implies that
    $$\lim_{k}\norm{x_{n_k}+y}=\norm{y+x}+d$$
   for each $y\in  \span(\{x\}\cup\{x_{n}\setsep n\in\en\})$.
   Then sequence $(x_{n_k}-x)$ satisfies the assumption of Lemma~\ref{L:m1-ell1} with constant $d$ (note that
   $\span(\{x_{n_k}-x,\setsep k\in\en\}\cup\{x\})=\span(\{x_{n_k},\setsep k\in\en\}\cup\{x\})$). Fix $\ep>0$. We shall prove that some weak$^*$ cluster point of $(x_{n_k})$ has norm at least $c(1-\ep)$. We will distinguish two subcases.

    Case 2a: Assume that $x=0$. Then Lemma~\ref{L:m1-ell1} provides a subsequence $(x_{n_{k_l}})$ such that $(\frac{x_{n_{k_l}}}d)$ is $\ep$-isometrically equivalent to the canonical basis of $\ell_1$. It follows that each weak$^*$-cluster point of $(x_{n_{k_l}})$ has norm at least $(1-\ep)d\ge(1-\ep)c$, which completes the argument. 

    Case 2b: Assume that  $x\ne0$.  Then Lemma~\ref{L:m1-ell1} provides a subsequence $(x_{n_{k_l}})$ such that the sequence
    $$\frac{x}{\norm{x}},\frac{x_{n_{k_1}}-x}d,\frac{x_{n_{k_2}}-x}d,\frac{x_{n_{k_3}}-x}d,\dots$$ is $\ep$-isometrically equivalent to the canonical basis of $\ell_1$. It follows from Lemma~\ref{L:hrombod} that each cluster point of $(x_{n_{k_l}})=(x+x_{n_{k_l}}-x)$ has norm at least
    $$(1-\ep)(\norm{x}+d)=(1-\ep)(\norm{x}+\lim\norm{x_{n_{k}}-x})\ge(1-\ep)\limsup \norm{x_{n_{k}}}\ge(1-\ep)c,$$
    which completes the argument.

    In both cases 2a and 2b we get $\dh(\clu{x_n},\{0\})\ge(1-\ep)c$. Since $\ep>0$ was arbitrary, we deduce that 
    $\dh(\clu{x_n},\{0\})\ge c$.

    \smallskip

It remains to observe that in both cases 1 and 2 we proved $\dh(\clu{x_n},\{0\})\ge c$. Since $c$ was arbitrary, we conclude that $\dh(\clu{x_n},\{0\})\ge\limsup\norm{x_n}$. Hence, condition $(i_1)$ from Proposition~\ref{P:schur} is fulfilled. The same proposition then says that $X$ has the $1$-Schur property. 
    \smallskip

    Let us continue by proving the $1$-strong Schur property. We note that for real spaces this follows from the first statement using Proposition~\ref{P:strongschur}. But the complex case does not follow directly, so we provide a proof.

    Fix $\delta\in(0,2]$ and $\ep>0$. Let $(x_k)$ be a normalized $\delta$-separated sequence in $X$. Let $x\in X$ and a subsequence $(x_{n_k})$ be provided by the assumption. By passing to a further subsequence we may assume that $\norm{x_{n_k}-x}\to d\in[0,\infty)$. Since $(x_n)$ is $\delta$-separated, the triangle inequality yields $d\ge\frac\delta2>0$. Hence, $(x_{n_k}-x)$ and $d$ satisfy the assumption of Lemma~\ref{L:m1-ell1} (see Case 2 above). Lemma~\ref{L:m1-ell1} provides a subsequence $(x_{n_{k_l}})$ such that the sequence
    $$\frac{x_{n_{k_1}}-x}d,\frac{x_{n_{k_2}}-x}d,\frac{x_{n_{k_3}}-x}d,\dots$$ is $\ep$-isometrically equivalent to the canonical basis of $\ell_1$, i.e.,
    $$(1-\ep)\sum_{l=1}^m\abs{\alpha_l}\le \norm{\sum_{l=1}^m\alpha_l\tfrac{x_{n_{k_l}}-x}d}\le (1+\ep)\sum_{l=1}^m\abs{\alpha_l}$$ for each $m\in\en$ and $\alpha_1,\dots,\alpha_m\in\ef$. It follows that
     $$\norm{\sum_{l=1}^m\alpha_l(x_{n_{k_l}}-x)}\ge d(1-\ep)\sum_{l=1}^m\abs{\alpha_l}$$
    for each $m\in\en$ and $\alpha_1,\dots,\alpha_m\in\ef$. I.e., sequence $(x_{n_{k_l}}-x)$ satisfies lower $\ell_1$-estimates with constant $d(1-\ep)$. By \cite[Lemma 4.2]{knaust-odell} we may assume, up to omitting first few elements of the sequence, that sequence $(x_{n_{k_l}})$ satisfies lower $\ell_1$-estimates with constant $d(1-\ep)$ as well. (We note that the quoted lemma is stated and proved for real spaces, but the same proof works in the complex case as well.)
    Since $(x_n)$ is a normalized sequence, we deduce that $(x_{n_{k_l}})$ is $\frac{1}{d(1-\ep)}$-equivalent to the $\ell_1$-basis. Since $d\ge\frac\delta2$, we get
    $\frac{1}{d(1-\ep)}\le\frac2{\delta(1-\ep)}$. Since $\ep>0$ is arbitrary, the assertion follows.  
    \end{proof}

We continue by a sufficient condition using property $(m_1)$ with respect to some linear topology.

     \begin{thm}\label{T:m1-tau}
         Let $X$ be a Banach space. Assume that there is a linear topology $\tau$ on $X$ satisfying the following two conditions:
         \begin{enumerate}[$(i)$]
             \item Any bounded sequence in $X$ admits a $\tau$-convergent subsequence.
             \item If $(x_n)$ is a bounded sequence $\tau$-converging to $0$, then
             $$\forall y\in X\colon \limsup\norm{x_n+y}=\norm{y}+\limsup\norm{x_n}.$$
         \end{enumerate}
         Then $X$ has both the $1$-Schur property and the $1$-strong Schur property.
     \end{thm}
    
   \begin{proof}
       We will show that the assumptions of Proposition~\ref{P:subs-1Schur} are fulfilled. Let $(x_n)$ be any bounded sequence in $X$. By $(i)$ we may find a subsequence $(x_{n_k})$ $\tau$-converging to some $x\in X$. Since $\tau$ is a linear topology, $(x_{n_k}-x)$ $\tau$-converges to $0$ (together with any further subsequence). Now it is enough to apply condition $(ii)$ to $(x_{n_k}-x)$ and $x+y$.
   \end{proof}

\begin{cor}
    Let $X$ be a Banach space with property $(m_1^*)$ such that $B_{X^*}$ is weak$^*$-sequentially compact.
    Then $X^*$ has both the $1$-Schur property and the $1$-strong Schur property. 

    The assumptions are satisfied, in particular, if $X$ is isometric to a subspace of $c_0(\Gamma)$.
\end{cor}

\begin{proof}
    The first statement follows from Theorem~\ref{T:m1-tau} applied to $\tau=w^*$. To show the `in particular' part assume that $X$ is a subspace of $c_0(\Gamma)$. Then $X$ has property $(m_1^*)$ by \cite[Lemma 1.2]{qschur-dp}. Further, $B_{X^*}$ is known to be weak$^*$-sequentially compact. A possible argument is the following:
    The space $c_0(\Gamma)$ is weakly compactly generated (see, e.g., \cite[Example (iii) on p. 575]{fhhmpz}), so its dual unit ball (in the weak$^*$ topology) is an Eberlein compact space by \cite[Theorem 13.20]{fhhmpz}. Thus it is sequentially compact (by the Eberlein-\v{S}mulyan theorem). Finally, since $X$ is a subspace of $c_0(\Gamma)$, $(B_{X^*},w^*)$ is a continuous image of $(B_{c_0(\Gamma)^*},w^*)$, so it is sequentially compact as well.
\end{proof}

\begin{remarks}\label{rem:m1}
    (1) The above corollary provides an alternative proof of \cite[Theorem 1.1]{qschur-dp}. The original proof uses property $(m_1^*)$ and also the concrete structure of $c_0(\Gamma)$ and its subspaces. The current section shows that concrete structure of $c_0(\Gamma)$ is not needed -- property $(m_1^*)$ together with weak$^*$ sequential compactness of the dual unit ball is enough.

    However, this does not make the proof `easy'. The reason is that the proof of property $(m_1^*)$ is nontrivial, it is based on a result from \cite{kalton-werner}.

    (2) In this section we established a sufficient condition for the $1$-Schur property. Let us stress that the constant $1$ is essential. Indeed, we cannot easily adapt this condition to yield the $c$-Schur property for some $c>1$. To see the reason for that assume that a Banach space $(X,\norm{\cdot})$ satisfies assumptions of Proposition~\ref{P:subs-1Schur} and assume that $\norm{\cdot}_n$ is an equivalent norm on $X$. Then $(X,\norm{\cdot}_n)$ has the quantitative Schur property, but possibly not $1$-Schur property. We further get that
    the equality from Proposition~\ref{P:subs-1Schur} is replaced by an inequality of the form
    $$\limsup_{l}\norm{x_{n_{k_l}}+y}_n\ge c(\norm{y+x}_n+\limsup_l
    \norm{x_{n_{k_l}}-x}_n).$$
    But a condition of this form does not imply the Schur property at all, as for any $p\in (1,\infty)$ the space $\ell_p$ has property $(m_p)$, i.e., for any $x\in\ell_p$ and any weakly null sequence in $\ell_p$ we have
    $$\limsup\norm{x_n+x}=\left(\norm{x}^p+\limsup\norm{x_n}^p\right)^{1/p}\ge 2^{\frac1p-1}(\norm{x}+\limsup\norm{x_n}).$$

    (3) Our sufficient condition yield both $1$-Schur property and $1$-strong Schur property. We recall that for real spaces the $1$-Schur property implies the $1$-strong Schur property (by Proposition~\ref{P:strongschur}). However, for complex spaces Example~\ref{ex:ell_1-complexification} provides a complex space with the $1$-Schur property which fails the $1$-strong Schur property. Thus the condition in Proposition~\ref{P:subs-1Schur} is not necessary for the $1$-Schur property. Moreover, the space from Example~\ref{ex:ell_1-complexification} is clearly a dual space, so its predual witnesses that there is a space without property $(m_1^*)$ whose dual has the $1$-Schur property.

    (4) Assume that $X$ is a complex Banach space and denote by $X_R$ its real version. Then $X$ has the $1$-Schur property if and only if $X_R$ has this property. Thus, if we replace in Proposition~\ref{P:subs-1Schur} the linear span by the real-linear span, we obtain a sufficient condition for the $1$-Schur property.
\end{remarks}

\section{Preservation of quantitative Schur property}\label{sec:preserving}

In this section we study preservation of the quantitative Schur property to direct sums. We start by a general result on finite direct sums.

\begin{thm}\label{T:konecne}
    Let $N\in\en$, let $X_1,\dots,X_N$ be Banach spaces satisfying the $c$-Schur property (for some $c\ge1$).
    Let $\Phi$ be a norm on $\er^N$ such that $\norm{\cdot}_\infty\le \Phi\le\norm{\cdot}_1$. 
    Let $Y=X_1\times\dots\times X_N$ be equipped with the norm
    $$\norm{(x_1,\dots,x_N)}=\Phi((\norm{x_1},\dots,\norm{x_N})),\quad (x_1,\dots,x_N)\in Y.$$
    Then $Y$ has the $c$-Schur property as well.
\end{thm}

\begin{proof}
    Clearly $Y$ is a Banach space. Further, without loss of generality
    $$\Phi((t_1,\dots,t_N))=\Phi((\abs{t_1},\dots,\abs{t_N})\mbox{  for }(t_1,\dots,t_N)\in\er^N.$$ Let $\Phi^*$ denote the dual norm.

     We will check that $Y$ fulfills condition $(i_c)$ from Proposition~\ref{P:schur}. Let $(x^k)$ be a bounded sequence in $Y$ such that $\limsup\norm{x^k}>\alpha>0$. Let
     $$z^k=(\norm{x^k_1},\dots,\norm{x^k_N})\in\er^N \mbox{ for }k\in\en.$$ Then $\limsup \Phi(z^k)>\alpha$.
     Up to passing to a subsequence we may assume $\Phi(z^k)>\alpha$ for each $k\in\en$
     and that $z^k\to z\in\er^N$. Then clearly $\Phi(z)\ge\alpha$.

     Let $\ep>0$ be arbitrary.  For each $j=1,\dots,N$ we perform inductively the following construction:
     \begin{enumerate}[$(i)$]
         \item Set $x^{k,1}=x^k$.
         \item Given $(x^{k,j})_k$, let $(x^{k,j+1})_k$ be a subsequence of $(x^{k,j})_k$ and $x_j^*\in (X_j)^*$ be such that $\Re\ip{x_j^*}{x^{k,j+1}_j}\ge (1-\ep)\frac{z_j}{c}$  for each $k\in\en$.
     \end{enumerate}
     
     Step $(i)$ may be clearly done. Assume that $j\in\{1,\dots,N\}$ and that $(x^{k,j})$ is given. We have $\norm{x^{k,j}}\to z_j$. If $z_j>0$, property $(i_c)$ in $X_j$ provides $x^{**}_j$, a weak$^*$-cluster point of $(x^{k,j}_j)_k$ in $(X_j)^{**}$ such that $\norm{x^{**}_j}>(1-\ep)\frac{z_j}{c}$. Thus there is some $x_j^*\in (X_j)^*$ of norm one such that $\Re\ip{x^{**}_j}{x^*_j}>(1-\ep)\frac{z_j}{c}$. Then $\Re\ip{x_j^*}{x^{k,j}_j}\ge (1-\ep)\frac{z_j}{c}$ for infinitely many $k$, thus we may find a further subsequence $(x^{k,j+1})$ with the required property.
     If $z_j=0$, the step is trivial.

     Next set $y^k=x^{k,N+1}$ for $k\in\en$. Then $(y^k)$ is a subsequence of $(x_k)$ and we have
     $$\Re\ip{x_j^*}{y^{k}_j}\ge (1-\ep)\frac{z_j}{c}\mbox{ for each }k\in\en, j\in\{1,\dots,N\}.$$
     Finally, fix $(s_1,\dots,s_N)\in \er^N$ with nonnegative coordinates such that $\Phi^*(s_1,\dots,s_N)=1$ and
     $$\sum_{j=1}^N s_j z_j=\Phi(z).$$
     Define an element of $\psi\in Y^*$ by 
     $$\psi((x_1,\dots,x_N))=\sum_{j=1}^N s_j x^*_j(x_j),\quad (x_1,\dots,x_N)\in Y.$$
     Then $\norm{\psi}=1$ and for each $k\in\en$ we have
     $$\Re \psi(y^k)=\sum s_j\Re\ip{x_j^*}{y^{k}_j}\ge\tfrac{1-\ep}c \sum_{j=1}^N s_j z_j=\tfrac{1-\ep}c \Phi(z)\ge \frac\alpha c(1-\ep).$$
     Thus each weak$^*$-cluster point of $(y^k)$ has norm at least $\frac\alpha c(1-\ep)$. Since $\ep>0$ is arbitrary, the proof of property $(i_c)$ of $Y$ is complete.
\end{proof}

\begin{cor}\label{cor:finite}
    The $c$-Schur property is preserved by finite $\ell_p$-sums for each $p\in[1,\infty]$.
\end{cor}

We continue by showing limits of preservation of the quantitative Schur property by finite direct sums.

\begin{prop}
Let $X_1,\dots,X_n$ be Banach spaces satisfying the $c$-Schur property for some $c\ge 1$. Let $X=X_1\times\dots\times X_n$ equipped with a norm $\norm{\cdot}$ satisfying
$$\max\{\norm{x_j}\setsep 1\le j\le n\}\le \norm{(x_1,\dots,x_n)}\le \sum_{j=1}^n\norm{x_j},\quad (x_1,\dots,x_n)\in X.$$
Then $X$ satisfies the $nc$-Schur property. Moreover, the factor $n$ is optimal.   
\end{prop}

\begin{proof}
Set $\norm{(x_1,\dots,x_n)}_\infty=\max  \{\norm{x_j}\setsep 1\le j\le n\}$ for $(x_1,\dots,x_n)\in X$. Then $(X,\norm{\cdot}_\infty)$ satisfies the $c$-Schur property by Corollary~\ref{cor:finite}. Moreover, our assumptions imply
$$\norm{\cdot}_\infty\le \norm{\cdot}\le n\cdot\norm{\cdot}_\infty.$$
We now easily deduce that $(X,\norm{\cdot})$ satisfies the $nc$-Schur property.

Let us now look at the optimality of $n$. If $n=1$, the optimality is obvious. We further fix $n\in\en$ and construct an example showing that for $n+1$ spaces the factor $n+1$ is optimal. 

Set $X_0=\ef$ and $X_1=X_2=\dots=X_{n}=\ell_1$. These spaces satisfy the $1$-Schur property. Let us define a norm on $X=X_0\times\dots\times X_n$ by
$$\begin{aligned}
    \norm{(t,x_1,\dots,x_n)}=\max\Big\{&\norm{x_n}_1,\norm{x_{n-1}}_1+\norm{x_n}_\infty,\\&\norm{x_{n-2}}_1+\norm{x_{n-1}}_\infty+\norm{x_n}_\infty,\dots,\abs{t}+\sum_{j=1}^n\norm{x_j}_\infty\Big\}.\end{aligned}$$
It is obvious that this formula defines a seminorm. Moreover, clearly
$$\max\{\abs{t},\max\{\norm{x_j}_1\setsep 1\le j\le n\}\}\le \norm{(t,x_1,\dots,x_n)}\le \abs{t}+\sum_{j=1}^n\norm{x_j}_1,$$
so it is a norm satisfying our assumption.

For $k\in \en$ set
$$x^k=(1,e_k,\dots,e_k)\in X,$$
where $e_k$ is the $k$-th canonical vector in $\ell_1$. Then $\norm{x^k}=n+1$ for each $k\in\en$. We claim that each weak$^*$-cluster point $x^{**}$ of $(x^k)$ in $X^{**}$ has norm at most $1$. If we prove this, optimality of the constant follows immediately.

We proceed by contradiction. Assume that $x^{**}$ is such a cluster point and that $\norm{x^{**}}>1+\ep$ for some $\ep>0$. Find $x^*\in X^*$ of norm one such that $\Re\ip{x^{**}}{x^*}>1+\ep$. Then there is an increasing sequence $(k_l)$ of natural numbers
$$\Re \ip{x^*}{x^{k_l}}>1+\ep\mbox{ for each }l\in\en.$$
The functional $x^*$, being an element of $X^*$, is canonically represented by an element $(s,y_1,\dots,y_n)\in \ef\times (\ell_\infty)^n$. So,
$$\Re (s+\sum_{j=1}^n y_{j,k_l}) >1+\ep \mbox{ for each }l\in\en.$$

Let $m\in\en$. Set 
$$u^m=\frac1m\sum_{l=1}^m e_{k_l}\in\ell_1$$
and
$$z^m=((1-\tfrac1m)^n,(1-\tfrac1m)^{n-1}u^m,\dots,(1-\tfrac1m)u^m,u^m)\in X.$$
We claim that $\norm{z^m}=1$. Indeed, note that $\norm{u^m}_1=1$ and $\norm{u^m}_\infty=\frac1m$ and
$$(1-\tfrac1m)^{k} +\tfrac1m \sum_{j=0}^{k-1}(1-\tfrac1m)^j =(1-\tfrac1m)^{k} +\tfrac1m \cdot \frac{1-(1-\tfrac1m)^k}{1-(1-\tfrac1m)}=1$$
for each $k\in\{1,\dots,n\}$.

Hence
$$\begin{aligned}
 1&\ge \Re \ip{x^*}{z^m} = \Re \left( s(1-\tfrac{1}{m})^n + \sum_{j=1}^n (1-\tfrac1m)^{j-1} \sum_{l=1}^m  \tfrac{y_{j,k_l}}{m}\right) 
 \\&=\Re \left(\sum_{l=1}^m \tfrac1m(s+\sum_{j=1}^n y_{j,k_l})  + s((1-\tfrac1m)^n-1) + \sum_{j=1}^n ((1-\tfrac1m)^{j-1}-1) \sum_{l=1}^m  \tfrac{y_{j,k_l}}{m}\right)
 \\& > 1+\ep - \abs{s}(1-(1-\tfrac1m)^n)- \sum_{j=1}^n (1-(1-\tfrac1m)^{j-1})\norm{y_j}_\infty
 \\&\ge 1+\ep - (1-(1-\tfrac1m)^n)- \sum_{j=1}^n (1-(1-\tfrac1m)^{j-1}]) >1
\end{aligned}$$
if $m$ is large enough. This contradiction completes the argument.

\end{proof}

We continue by preservation by arbitrary $\ell_1$-sums.

\begin{thm}\label{T:ell1suma}
    Let $c\ge 1$. Then the $c$-Schur property is preserved by arbitrary $\ell_1$-sums.
\end{thm}

\begin{proof} 
    The case of finite $\ell_1$-sums is covered by Corollary~\ref{cor:finite}, so we focus on infinite sums.
    Since the $c$-Schur property is defined via sequences and each element in the $\ell_1$-sum of a family of Banach spaces has countable support, it is enough to prove the statement for countable $\ell_1$-sums. 
    So, assume that $(X_n)$ is a sequence of Banach spaces satisfying the $c$-Schur property and let $X$ be their $\ell_1$-sum. We will proceed by a modification of the proof of \cite[Theorem 1.3]{qschur}.
    
    Let $(x^k)$ be a bounded sequence in $X$ such that $\limsup\norm{x^k}>\alpha>0$.  For $k\in\en$ set $y^k=(\norm{x^k_n})_{n\in\en}\in\ell_1$. Then
    $\limsup_k\norm{y_k}>\alpha$. Up to passing to a subsequence we may assume that $\norm{y_k}>\alpha$ for each $k\in\en$ and that $y_k\to y\in\ell_1$ coordinatewise. 
    
    Fix an arbitrary $\ep>0$ and choose $m\in\en$ such that $\norm{y|_{(m,\infty)}}<\ep$. Then $\norm{x^k|_{[1,m]}}=\norm{y^k|_{[1,m]}}\to \norm{y|_{[1,m]}}$. Hence we may find $x_0^*\in (\bigoplus_{1\le n\le m} X_n)^*$ of norm one such that
    $$\Re\ip{x_0^*}{x^k|_{[1,m]}}\ge\tfrac1c(1-\ep)\norm{y|_{[1,m]}}\mbox{ for infinitely many }k\in\en.$$
    Indeed, this is trivial if $y|_{[1,m]}=0$. If $y|_{[1,m]}\ne0$, we use Corollary~\ref{cor:finite} to find $x_0^{**}$, a weak$^*$-cluster point of $(x^k|_{[1,m]})$ in the bidual such that $\norm{x_0^{**}}>\frac1c(1-\ep)\norm{y|_{[1,m]}}$. Then there is $x_0^*$ of norm one such that $\Re\ip{x_0^{**}}{x_0^*}>\frac1c(1-\ep)\norm{y|_{[1,m]}}$. Such a functional has clearly the required property.

    Up to passing to a subsequence we may assume that the inequality holds for all $k\in\en$.
    Next we apply \cite[Lemma 2.1]{qschur} to sequence $(y^k|_{(m,\infty)})$ and we get natural numbers $m=N_0<N_1<N_2<\cdots$  and a subsequence $(y^{n_k})$ such that
    $$\norm{y^{n_k}|_{(N_{k-1},N_k]}}>\norm{y^{n_k}|_{(m,\infty)}}-\ep\mbox{ for each }k\in\en.$$

    If $k\in\en$ and $j\in (N_{k-1},N_k]$, we find a norm-one functional $x_j^*\in (X_j)^*$ such that
    $\ip{x_j^*}{x^{n_k}_j}=\norm{x^{n_k}_j}$.

    Next we are going to define a suitable norm-one element $x^*\in X^*$. Note that $X^*$ is canonically isometric to the $\ell_\infty$-sum of dual spaces $X_n^*$, so we may define $x^*$ such that $x^*|_{[1,m]}=x_0^*$ and that the $j$-th coordinate equals $x^*_j$ whenever $j>m$.
    Then we have for each $k\in\en$ 
    $$\begin{aligned}
        \Re\ip{x^*}{x^{n_k}}&=\Re\ip{x_0^*}{x^{n_k}|_{[1,m]}} + \sum_{l=1}^\infty \sum_{j\in (N_{l-1},N_l]} \Re\ip{x^*_j}{x^{n_k}_j}
        \\&\ge \tfrac1c(1-\ep)\norm{y|_{[1,m]}} +\sum_{j\in (N_{l-1},N_l]} \norm{x^{n_k}_j} + \sum_{l\in\en\setminus\{k\}} \sum_{j\in (N_{l-1},N_l]} \Re\ip{x^*_j}{x^{n_k}_j}
        \\&\ge  \tfrac1c(1-\ep)\norm{y|_{[1,m]}} + \norm{y^{n_k}|_{(N_{k-1},N_k]}} - \norm{y^{n_k}|_{(m,\infty)\setminus(N_{k-1},N_k]}}
        \\&=\tfrac1c(1-\ep)\norm{y|_{[1,m]}} +2 \norm{y^{n_k}|_{(N_{k-1},N_k]}} - \norm{y^{n_k}|_{(m,\infty)}}
        \\&>\tfrac1c(1-\ep)\norm{y|_{[1,m]}}+\norm{y^{n_k}|_{(m,\infty)}}-2\ep
    \end{aligned}$$
    Since $y^{n_k}|_{[1,m]}\to y|_{[1,m]}$, we get that for all $k\in\en$ with finitely many exceptions we have
    $$\begin{aligned}         
    \Re\ip{x^*}{x^{n_k}}&>\tfrac1c(1-\ep)\norm{y^{n_k}|_{[1,m]}}+\norm{y^{n_k}|_{(m,\infty)}}-2\ep
    \ge \tfrac1c(1-\ep)\norm{y^{n_k}}-2\ep\\&>\tfrac\alpha{c}(1-\ep)-2\ep.\end{aligned}$$
    Hence any weak$^*$-cluster point of $(y^{n_k})$ has norm at least $\tfrac\alpha{c}(1-\ep)-2\ep$. Since $\ep>0$ was arbitrary, the argument is complete.   
\end{proof}

\section{Schur property of Lipschitz-free spaces}\label{sec:LF}

In  \cite[Theorem C]{p1u}  it is proved that the Lipschitz-free space over a complete metric space $M$ has the Schur property if and only if $M$ is purely $1$-unrectifiable. If $M$ is additionally proper, i.e., closed bounded sets are compact, then the respective free space has even $1$-Schur property. This follows by combining \cite[Proposition 17]{petitjean-schur} with \cite[Theorem 3.2]{p1u}. We will address a special subclass of purely $1$-unrectifiable spaces formed by uniformly separated metric spaces. In particular, we provide an example distinguishing the $1$-Schur property and the $1$-strong Schur property in real setting.

Let us briefly recall the definition of a Lipschitz-free space and basic properties of such spaces. Let $(M,d)$ be a metric space with a distinguished point $0\in M$. By $\Lip_0(M)$ we denote the space of all real-valued Lipschitz functions on $M$ which vanish at the distinguished point $0$. If we equip this space by the least Lipschitz constant norm, it becomes a Banach space.
For $x\in M$ we denote by $\delta(x)$ the evaluation functional on $\Lip_0(M)$ (defined by $f\mapsto f(x)$). The Lipschitz-free space over $M$ is then defined by
$$\F(M)=\overline{\span}\{\delta(x)\setsep x\in M\}.$$

Basic properties of Lipschitz-free spaces which we will need are collected in the following proposition.

\begin{prop}\label{P:LF-basic}
    Let $(M,d)$ be a metric space with a distinguished point $0\in M$. Then the following assertions are valid.
    \begin{enumerate}[$(a)$]
        \item The mapping $\delta:M\to\F(M)$ is an isometric embedding.
        \item The dual space $\F(M)^*$ is canonically isometric to $\Lip_0(M)$. More precisely, the assignment
        $$\varphi\in \F(M)^*\mapsto ( x\in M\mapsto \varphi(\delta(x)) )$$
        is a surjective linear isometry of $\F(M)^*$ onto $\Lip_0(M)$.
        \item Let $\star\in M$ be another distinguished point and let $\Lip_\star(M)$ denote the space of real-valued Lipschitz functions on $M$ vanishing at $\star$ and let $\F_\star(M)$ denote the induced Lipschitz-free space. Then $\Lip_\star(M)$ is isometric to $\Lip_0(M)$ and $\F_\star(M)$ is isometric to $\F(M)$. More precisely,
        $f\mapsto f-f(0)$ is a surjective linear isometry of $\Lip_\star(M)$ onto $\Lip_0(M)$ and the assignment
        $$\delta_\star(x)\mapsto \delta(x)-\delta(\star), \quad x\in M,$$
        uniquely extends to a surjective linear isometry of $\F_\star(M)$ onto $\F(M)$.
        \item If $N$ is a (metric) subspace of $M$, then $\F(N)$ is isometric to a linear subspace of $\F(M)$. More precisely, if $0\in N$ and $0$ is used as the distinguished point of $N$, then the assignment
        $$\delta_N(x)\mapsto \delta_M(x), \quad x\in N,$$
        uniquely extends to a linear isometric injection of $\F(N)$ into $\F(M)$.
    \end{enumerate}
\end{prop}

\begin{proof}
    $(a)$: This is well known and easy to see. Note that the norm on $\F(M)$ is inherited from $\Lip_0(M)^*$. Hence 
    $$\norm{\delta(x)-\delta(y)}=\sup\{\abs{f(x)-f(y)}\setsep f\in\Lip_0(M),\norm{f}\le 1\}=d(x,y).$$
    Indeed, inequality `$\le$' is obvious, the converse one is witnessed by the choice $f(t)=d(t,x)-d(0,x)$.

    Assertions $(b)$--$(d)$ easily follow from the universal property of Lipschitz-free spaces explained and proved in \cite[Section 2]{cuth-doucha-w}. In fact, assertions $(b)$ and the special case of $(d)$ (when we assume $0\in N$) are proved in \cite[Section 2]{cuth-doucha-w}. Assertion $(c)$ follows from the universal property applied to the mapping $x\mapsto \delta(x)-\delta(\star)$ and the general case of $(d)$ follows from the special case using $(c)$.
\end{proof}

We continue by an easy proposition on bounded unifromly separated metric spaces.

\begin{prop}\label{P:discrete-ab}
    Let $M$ be a bounded uniformly separated metric space. Then $\F(M)$ has the quantitative Schur property. 
    More precisely, assume that $0<a\le b<\infty$ and
    $$a\le d(x,y)\le b\mbox{ whenever }a,b\in M, a\ne b.$$
    Then $\F(M)$ has the $\frac ba$-Schur property.
\end{prop}

\begin{proof} Let us define an extension of $M$ by setting $M^\prime=M\cup\{\star\}$ and extending the metric $d$ by
    $$d(x,\star)=\tfrac b2, \quad x\in M.$$
    It is clear that in this way we obtain a metric space. Since $\F(M)$ is isometric to a subspace of 
    $\F(M^\prime)$ (by Proposition~\ref{P:LF-basic}$(d)$), it is enough to prove that $\F(M^\prime)$ has the $\frac ba$-Schur property.

    Let $\star$ be the distinguished point. Then
    a function $f:M^\prime\to\er$ belongs to $\Lip_0(M)$ if and only if $f(\star)=0$ and $f|_M$ is bounded. Hence
    $f\mapsto f|_M$ is a linear bijection of $\Lip_0(M^\prime)$ onto $\ell_\infty(M)$. Let us estimate the Lipschitz norm. Fix $f\in \Lip_0(M^\prime)$. Then
    $$L(f)\ge\sup_{x\in M} \frac{\abs{f(x)-f(\star)}}{d(x,\star)}=\frac 2b\norm{f|_M}_\infty.$$
    To find an upper bound fix $x,y\in M^\prime$ distinct. If $x,y\in M$, then
    $$\frac{\abs{f(x)-f(y)}}{d(x,y)}\le \frac{2\norm{f|_M}_\infty}{a}.$$
    If one of them, say $y$, equals $\star$, then
    $$\frac{\abs{f(x)-f(y)}}{d(x,y)}\le \frac2b \norm{f|_M}_\infty\le \frac2a \norm{f|_M}_\infty.$$
    Hence
    $$\frac2b\norm{f|_M}_\infty\le L(f)\le \frac2a\norm{f|_M}_\infty.$$

    It easily follows from the duality of $\F(M^\prime)$ and $\Lip_0(M^\prime)$ that $\F(M^\prime)$ is canonically isomorphic to $\ell_1(M)$ and for each $\mu\in \F(M^\prime)$ we have
    $$\frac a2\norm{\mu}_1\le \norm{\mu}\le \frac b2\norm{\mu}_1.$$
    Since $\ell_1(M)$ has the $1$-Schur property, we easily deduce that $\F(M)$ has the $\frac ba$-Schur property.
\end{proof}

We continue by the promised example.

\begin{example}\label{ex:LF}
There is a metric space $(M,d)$ with the following properties:
\begin{enumerate}[$(a)$]
    \item $M=\zet$ and $d$ is the shortest-path distance for a graph on $\zet$.
    \item $M$ is $1$-separated and $\diam M=2$.
    \item $\F(M)$ has the $2$-Schur property but it fails the $c$-Schur property for each $c<2$.
    \item $\F(M)$ has the $1$-strong Schur property.
\end{enumerate}
    
\end{example}

\begin{proof} The proof will be done in several steps.

\smallskip

\noindent{\tt Step 1:} Construction of $M$ and proof of easy statements:

\smallskip

 Let $G=(\zet,E)$ be the graph with $\zet$ as the set of vertices such that $(m,n)\in E$ (i.e., $m$ and $n$ are joined by an edge) if and only if $m\ne\pm n$. Let $d$ be the shortest-path distance. Then
 $$d(m,n)=\begin{cases} 0 & m=n, \\ 2 & m=-n\ne0,\\ 1 &\mbox{otherwise}.\end{cases}$$
Then clearly $(M,d)$ satisfies $(a)$ and $(b)$. Further, $\F(M)$ has the $2$-Schur property by Proposition~\ref{P:discrete-ab}.

\smallskip

\noindent{}{\tt Step 2:} A function $f:M\to\er$ satisfying $f(0)=0$ is $1$-Lipschitz if and only if one of the following two conditions is satisfied:
\begin{enumerate}[$(1)$]
    \item There is some $c\in [0,1]$ such that $f(\zet)\subset [c-1,c]$.
    \item There are $n\in\zet\setminus\{0\}$, $a,b\in (0,1]$ with $a+b>1$ such that
    $$f(n)=a,f(-n)=-b,f(\zet\setminus\{-n,n\})\subset [a-1,1-b].$$ 
\end{enumerate}

\smallskip

Indeed, if $(1)$ is satisfied, then $f$ is $1$-Lipschitz as $M$ is $1$-separated. If $(2)$ is fulfilled,
then $\abs{f(n)-f(-n)}\le 2=d(n,-n)$ and if $k,l\in \zet$ such that  $k\ne l$ and at least one of them does not belong to $\{n,-n\}$, then $\abs{f(k)-f(l)}\le 1\le d(k,l)$. Hence $f$ is $1$-Lipschitz.

Conversely assume that $f$ is $1$-Lipschitz. Since $f(0)=0$ and $d(n,0)=1$ for each $n\ne 0$, we get
$0\in f(\zet)\subset [-1,1]$. Assume that $(1)$ is not satisfied. Then there are $m,n\in\zet$ such that $\abs{f(m)-f(n)}>1$. Then necessarily $d(m,n)>1$, thus $m=-n\ne0$. Up to relabeling $m$ and $n$ we may assume that $a=f(n)>0$ and $b=-f(-n)>0$. For each $k\in \zet\setminus\{n,-n\}$ we have $d(k,n)=d(k,-n)=1$ and hence
$$1-b\le f(k)\le a-1.$$

\smallskip

\noindent{}{\tt Step 3:} $\F(M)$ fails the $c$-Schur property for each $c<2$.

\smallskip

Let $(x_k)$ be the sequence in $\F(M)$ given by
$$\delta(1),\delta(-1),\delta(2),\delta(-2),\delta(3),\delta(-3),\dots$$
Since for each $n\in\zet\setminus\{0\}$ we have $\norm{\delta(n)}=d(n,0)=1$, we get $\norm{x_k}=1$ for each $n\in\en$. Further, $\norm{\delta(n)-\delta(-n)}=d(n,-n)=2$, so $\norm{x_{2k}-x_{2k-1}}=2$ for each $k\in\en$.
We deduce that $\ca{x_k}=2$.

On the other hand, for each $1$-Lipschitz $f\in \Lip_0(M)$ we have
$$\begin{aligned}
    \ca{\ip{f}{x_k}}&=\limsup \ip{f}{x_k}-\liminf \ip{f}{x_k}
    \\ &\le \begin{cases}
     c-(c-1)=1 & \mbox{ if $f$ satisfies }(1),\\
     (1-b)-(a-1) =2-a-b<1 &\mbox{ if $f$ satisfies }(2).
\end{cases}\end{aligned}$$
Thus $\de{x_k}\le 1$. We conclude that $\F(M)$ fails the $c$-Schur property for $c<2$.

\smallskip

\noindent{\tt Step 4:} $\Lip_0(M)$ is isomorphic to $\ell_\infty(\zet\setminus\{0\})$ and $\F(M)$ is canonically isomorphic to $\ell_1(\zet\setminus\{0\})$. More precisely, $f\mapsto f|_{\zet\setminus\{0\}}$ is a linear bijection of $\Lip_0(M)$ onto $\ell_\infty(\zet\setminus\{0\})$ and
$$\norm{f|_{\zet\setminus\{0\}}}_\infty\le L(f)\le 2 \norm{f|_{\zet\setminus\{0\}}}_\infty,\quad f\in \Lip_0(M).$$
Further, $\F(M)$ is canonically isomorphic to $\ell_1(\zet\setminus\{0\})$. More precisely,
$$\frac12\sum_{n\in \zet\setminus\{0\}}\abs{c_n}\le \norm{\sum_{n\in\zet\setminus\{0\}} c_n\delta(n)}\le 2\sum_{n\in \zet\setminus\{0\}}\abs{c_n}$$
for each series with finitely many nonzero terms.

\smallskip

This follows from the proof of Proposition~\ref{P:discrete-ab}.

\smallskip

\noindent{\tt Step 5:} The norm on $\F(M)$ in the representation provided by Step 4 is given by $$\norm{x}=\max\{\norm{x^+}_1,\norm{x^-}_1,\max_{n\in\en} (\abs{x(n)}+\abs{x(-n)})\}, \quad x\in \ell_1(\zet\setminus\{0\}).$$

\smallskip

Let us start by proving inequality `$\ge$'. Fix $x\in\ell_1(\zet\setminus\{0\})$.
The function 
$$f(n)=\begin{cases}
    1 & x(n)\ge 0,\\ 0&\mbox{otherwise},
\end{cases}$$
is $1$-Lipschitz (by Step 2) and $\ip{f}{x}=\norm{x^+}_1$. Thus $\norm{x^+}_1\le \norm{x}$. Similarly we get $\norm{x^-}_1\le \norm{x}$. We further fix $n\in\en$. If $x(n)$ and $x(-n)$ are both in $[0,1]$ or both in $[-1,0]$, then
$$\abs{x(n)}+\abs{x(-n)}\le \max\{\norm{x^+}_1,\norm{x^-}_1\}\le\norm {x}.$$
Finally, assume that $x(n)$ and $x(-n)$ have (strictly) opposite signs. Then the function $g$ defined by
$$g(n)=\sign x(n),\ g(-n)=\sign x(-n),\ g=0\mbox{ elsewhere},$$
is $1$-Lipschitz (by Step 2) and $\ip{g}{x}=\abs{x(n)}+\abs{x(-n)}$. This completes the argument.

To prove inequality `$\le$' it is enough to take any $1$-Lipschitz function $f\in \Lip_0(M)$ and to prove that 
$\ip{f}{x}$ is bounded above by the expression on the right-hand side. So, let us take such $f$.

Assume first that $f$ satisfies $(1)$ from Step 2 and that $c\in[0,1]$ witnesses it. Then
$$\begin{aligned}
    \ip{f}{x}&=\sum_{n\in\zet\setminus\{0\}} f(n)x(n)= \sum_{x(n)>0} f(x) x(n)+ \sum_{x(n)<0} f(x) x(n)
\\&\le c\norm{x^+}_1+(1-c)\norm{x^-}_1\le\max\{\norm{x^+}_1,\norm{x^-}_1\}.\end{aligned}$$

Next assume that $f$ satisfies condition $(2)$ from Step 2 and that $n,a,b$ witness it. We shall distinguish several cases:

Case 1: $x(n)\le 0$: Define $\widetilde{f}$ by setting $\widetilde{f}(n)=0$ and $\widetilde{f}=f$ elsewhere. Then $\widetilde{f}$ is $1$-Lipschitz of type $(1)$ and we have
$$\ip{f}{x}\le\ip{\widetilde{f}}{x}\le\max\{\norm{x^+}_1,\norm{x^-}_1\}.$$

Case 2: $x(-n)\ge 0$:  Define $\widetilde{f}$ by setting $\widetilde{f}(-n)=0$ and $\widetilde{f}=f$ elsewhere. Then $\widetilde{f}$ is $1$-Lipschitz of type $(1)$ and we have
$$\ip{f}{x}\le\ip{\widetilde{f}}{x}\le\max\{\norm{x^+}_1,\norm{x^-}_1\}.$$

Case 3: $x(n)>0$ and $x(-n)<0$. Let $y\in \ell_1(\zet\setminus\{0\})$ be defined by $y(n)=y(-n)=0$ and $y=x$ elsewhere. Then
$$\begin{aligned}
\ip{f}{x}&=a x(n)-b x(-n)+ \sum_{k\in\zet\setminus \{0\}} f(k) y(k)
\\& =a \abs{x(n)} + b \abs{x(-n)}+ \sum_{y(k)>0} f(k) y(k) + \sum_{y(k)<0} f(k) y(k)
\\&\le a \abs{x(n)} + b \abs{x(-n)}+ (1-b)\norm{y^+}_1+(1-a)\norm{y^-}_1.
\end{aligned}$$

We now distinguish several subcases:

Subcase 3.1: $x(n)\ge \norm{y^-}_1$ and $x(-n)\le -\norm{y^+}_1$: 
Then $$\ip{f}{x} \le \abs{x(n)}+\abs{x(-n)}.$$

Subcase 3.2: $x(n)\ge \norm{y^-}_1$ and $0>x(-n)> -\norm{y^+}_1$:
Then
$$\ip{f}{x} \le \abs{x(n)}+\norm{y^+}_1 =\norm{x^+}_1.$$

Subcase 3.3: $0<x(n)<\norm{y^-}_1$ and $x(-n)\le -\norm{y^+}_1$:
Then
$$\ip{f}{x} \le \abs{x(-n)}+\norm{y^-}_1 =\norm{x^-}_1.$$

Subcase 3.4: $0<x(n)<\norm{y^-}_1$ and $0>x(-n)> -\norm{y^+}_1$:
Since $b>1-a$, we get 
$$\begin{aligned}
\ip{f}{x}&\le a \abs{x(n)} + b \abs{x(-n)}+ (1-b)\norm{y^+}_1+(1-a)\norm{y^-}_1
\\& \le  a \abs{x(n)} + (1-a) \abs{x(-n)}+ a\norm{y^+}_1+(1-a)\norm{y^-}_1
\\&= a\norm{x^+}_1+(1-a)\norm{x^-}_1\le \max\{\norm{x^+}_1,\norm{x^-}_1\}.
\end{aligned}$$
This completes the argument.  

\smallskip

\noindent{\tt Step 6:} Let $G^\prime$ be the full graph on $\zet$ and let $d^\prime$ be the respective shortest-path distance on $\zet$ (i.e., it is the discrete $0-1$ metric). Let $M^\prime=(\zet,d^\prime)$. Then:
\begin{enumerate}[$(i)$]
    \item $\Lip_0(M^\prime)$ is canonically isomorphic with $\ell_\infty(\zet\setminus\{0\})$ and the norm is given by $\norm{f}=\norm{f^+}_\infty+\norm{f^-}_\infty$, $f\in \ell_\infty(\zet\setminus\{0\})$. In particular, $1$-Lipschitz functions from $\Lip_0(M^\prime)$ are exactly functions satisfying property $(1)$ from Step 2.
    \item $\F(M^\prime)$ is canonically isomorphic with $\ell_1(\zet\setminus\{0\})$ and the norm is given by
    $\norm{x}=\max\{\norm{x^+}_1,\norm{x^-}_1\}$, $x\in\ell_1(\zet\setminus\{0\})$
    \item $\F(M^\prime)$ has the $1$-Schur property.
\end{enumerate}

\smallskip

Assertions $(i)$ and $(ii)$ are known and easy to prove. Assertion $(iii)$ follows from Proposition~\ref{P:discrete-ab}.

\smallskip

\noindent{\tt Step 7:} $\F(M)$ satisfies property $(iii_1)$ from Proposition~\ref{P:schur}. In particular, $\F(M)$ has the $1$-strong Schur property.

\smallskip

Let $(x_k)$ be a bounded sequence in $\F(M)$. Assume $\wca{x_k}>c>0$. By the classical Ramsey theorem we may assume, up to passing to a subsequence, that $\norm{x_k-x_l}>c$ for each $k\ne l$. Denote by $\norm{\cdot}^\prime$ the norm of $\F(M^\prime)$. Due to Step 6 it is an equivalent norm on $\F(M)$. We distinguish two cases:

Case I: $\ca[\F(M^\prime)]{x_k}\ge c$.

Case II: $\ca[\F(M^\prime)]{x_k}<c$.

Assume first that Case I occurs. By Step 6 we know that $\F(M^\prime)$ has the $1$-Schur property, thus $\delta_{\F(M^\prime)}(x_k)\ge c$. Since $B_{\Lip_0(M^\prime)}\subset B_{\Lip_0(M)}$ (by Steps 6 and 2), we deduce
$$\delta_{\F(M)}(x_k)\ge \delta_{\F(M^\prime)}(x_k)\ge c.$$
Since $\de{\cdot}$ decreases when passing to a subsequence, we conclude that $\de{x_k}\ge c$ also for the original sequence.

Next assume that Case II occurs. Fix $\ep>0$ such that $\ca[\F(M^\prime)]{x_k}<c-\ep$. Up to omitting finitely many elements we may assume that $\norm{x_k-x_l}^\prime<c-\ep$ whenever $k,l\in\en$, $k\ne l$. 
By Step 5 we deduce that
$$\forall k,l\in\en, k\ne l\;\exists n\in\en \colon \norm{(x_k-x_l)|_{\{n,-n\}}}_1>c.$$
Let us observe that for each $k\ne l$ there is a unique $n\in\en$ with the above property. Indeed, assume that for some $k\ne l$ the property is fulfilled by two different $n\in\en$. Then we get $\norm{x_k-x_l}_1>2c$. On the other hand, $\norm{x_k-x_l}_1\le 2\norm{x_k-x_l}^\prime< 2c$, a contradiction proving the uniqueness of $n$. We shall  denote the unique $n$ by $n(k,l)$.

We next obtain a decreasing sequence $(N_n)$ of infinite subsets of $\en$ such that
$$\forall n\in\en\;\forall k,l\in N_n, k\ne l\colon n(k,l)>n.$$
This sequence may be constructed by induction. Let $N_0=\en$ and assume we have constructed $N_{n-1}$.
By Ramsey theorem there is an infinite subset $N_n\subset N_{n-1}$ such that
either
$$\forall k,l\in N_n, k\ne l\colon \norm{(x_k-x_l)|_{\{n,-n\}}}_1\le c$$
or
$$\forall k,l\in N_n, k\ne l\colon \norm{(x_k-x_l)|_{\{n,-n\}}}_1> c.$$
But the second possibility cannot occur as $(x_k|_{\{n,-n\}})$ is a bounded sequence in a two-dimensional space.
Thus the first possibility takes place. This completes the inductive step.

Hence, we may construct a strictly increasing sequence $(k_n)$ of natural numbers such that $k_n\in N_n$ for each $n\in\en$.
Then $n(k_j,k_l)>\min\{j,l\}$ for each $j\ne l$. Up to relabeling the sequence we may assume that
$n(k,l)>\min\{k,l\}$ for each $k\ne l$.

Given $k\in\en$, $(n(k,l))_{l>k}$ is a sequence of natural numbers. Therefore it has either a constant subsequence or a strictly increasing subsequence. Therefore, by a standard inductive argument we may assume, up to passing to a subsequence, that
$$\forall k\in \en\colon \mbox{ the sequence $(n(k,l))_{l>k}$ is either constant or strictly increasing}.$$
Up to passing to a further subsequence we may assume that one of the following cases occurs:

Case A: The sequence $(n(k,l))_{l>k}$ is constant for each $k\in\en$.

Case B: The sequence $(n(k,l))_{l>k}$ is strictly increasing for each $k\in\en$.

Assume first that Case A takes place. Then there is a sequence $(m_k)$ of natural numbers such that $n(k,l)=m_k$ whenever $k<l$. Since $m_k\ge k\to\infty$, we may assume (up to passing to a subsequence) that $(m_k)$ is strictly increasing.

Given $k\in \en$, $(x_l|_{\{m_k,-m_k\}})_{l>k}$ is a bounded sequence in a two-dimensional space, so it has a convergent subsequence. So, up to passing again to a subsequence, we may assume that
\begin{itemize}
    \item For each $k\in \en$ the sequence $(x_l|_{\{m_k,-m_k\}})_{l>k}$ converges to some $z_k\in\ell_1(\{m_k,-m_k\})$;
    \item $\norm{z_k}_1\to \alpha\ge0$.
\end{itemize}

Assume that $\alpha>0$. Then we may assume, up to omitting finitely many elements, that $\norm{z_k}_1>\frac\alpha2$ for each $k\in\en$. Hence
$$\forall k\in\en\;\exists l_0\in\en\;\forall l\ge l_0\colon  \norm{x_l|_{\{m_k,-m_k\}}}_1>\tfrac\alpha2.$$

So, by passing to a further subsequence we may achieve that
$$\forall k\in\en\;\forall l>k \colon  \norm{x_l|_{\{m_k,-m_k\}}}_1>\tfrac\alpha2.$$
Since the sequence $(m_k)$ is strictly increasing, we deduce that $\norm{x_l}_1>(l-1)\cdot\frac\alpha2$ for each $l\in\en$, $l>1$. This is a contradiction with the assumption that $(x_n)$ is a bounded sequence.
Thus necessarily $\alpha=0$.

 Then there is $k_0\in\en$ such that for each $k\ge k_0$ we have $\norm{z_k}<\ep$. Let $k\ge k_0$ be arbitrary. We find $m\in\en$, such that $\sum_{n\ge m}(\abs{x_k(n)}+\abs{x_k(-n)})<\ep$. Further, let $l>k$ be such that $m_l>m$. Then
$$\norm{x_k-x_l}_1\ge \norm{(x_k-x_l)|_{\{m_k,-m_k\}}}_1+\norm{(x_k-x_l)|_{\{m_l,-m_l\}}}_1 >c+ \norm{x_l|_{\{m_l,-m_l\}}}_1-\ep.$$
Further, for each $p>l$ we have
$$\begin{aligned}
    \norm{x_l|_{\{m_l,-m_l\}}}_1&\ge \norm{(x_l-x_p)|_{\{m_l,-m_l\}}}_1 -\norm{x_p|_{\{m_l,-m_l\}}}_1 \\& > c- \norm{x_p|_{\{m_l,-m_l\}}}_1\overset{p}{\longrightarrow} c-\norm{z_l}>c-\ep.\end{aligned}$$
Hence we deduce that $\norm{x_k-x_l}_1>2c-2\ep$. Thus $\norm{x_k-x_l}^\prime>c-\ep$. This is a contradiction showing that Case A cannot occur. 

Finally, let us assume that Case B takes place. Fix  $k_1\in\en$. Further, let $M>0$ be such that $\norm{x_k}_1\le M$ for each $k\in\en$. Let us construct sequences $(k_j)$ and $(p_j)$ of natural numbers such that the following conditions are fulfilled.

\begin{itemize}
    \item $p_j>p_i$ for $1\le i<j$; 
    \item $\sum_{n\ge p_j}(\abs{x_{k_j}(n)}+\abs{x_{k_j}(-n)})<\frac\ep2$ for each $j\in \en$;
    \item $k_{j+1}>k_j$ for $j\in\en$;
    \item $n(k_j,l)>p_j$ whenever $j\in\en$ and $l\ge k_{j+1}$. 
\end{itemize}

It is clear that this construction may be performed. Then 
$$\norm{x_{k_{j+1}}|_{\{\pm n(k_j,k_{j+1})\}}}_1\ge \norm{(x_{k_{j+1}}-x_{k_j}))|_{\{\pm n(k_j,k_{j+1})\}}}_1-\norm{x_{k_{j}}|_{\{\pm n(k_j,k_{j+1})\}}}_1>c-\tfrac\ep2$$
for each $j\in\en$. We observe that the sequence $(n(k_j,k_{j+1}))_j$ is strictly increasing.
Fix $N\in\en$ such that $\frac MN<\frac\ep2$. It follows that there is some $j\in\{1,\dots,N\}$ such that
$\norm{x_{k_{N+2}}|_{\{\pm n(k_j,k_{j+1})\}}}_1<\frac\ep2$. Then
$$\begin{aligned}
   \norm{x_{k_{N+2}}-x_{k_j}}_1&\ge 
\norm{x_{k_{N+2}}|_{\{\pm n(k_{N+1},k_{N+2})\}}}_1-\norm{x_{k_{j}}|_{\{\pm n(k_j,k_{j+1})\}}}_1
\\&\qquad+\norm{x_{k_{j}}|_{\{\pm n(k_j,k_{j+1})\}}}_1-\norm{x_{k_{N+2}}|_{\{\pm n(k_j,k_{j+1})\}}}_1
\\&>c-\tfrac\ep2-\tfrac\ep2+c-\tfrac\ep2-\tfrac\ep2=2c-2\ep.\end{aligned}$$
Hence $\norm{x_{k_{N+2}}-x_{k_j}}^\prime>c-\ep$, a contradiction completing the proof. 
\end{proof}

\begin{example2}\label{ex:LF-complex}
Let $X$ be the complexification of the space $\F(M)$ from Example~\ref{ex:LF}. Then $X$ fails the $c$-Schur property for $c<2$, but fulfills condition $(iii_1)$ from Proposition~\ref{P:schur}. (This follows from Lemma~\ref{L:komplexifikace}.) Hence property $(iii_1)$ and the $1$-Schur property are different also within complex Banach spaces.
\end{example2}

\begin{example}\label{ex:LF-3}
    There is a metric space $(M,d)$ with the following properties:
    \begin{enumerate}[$(a)$]
        \item The metric $d$ is the shortest-path distance on a countable graph.
        \item The metric $d$ attains only values $0,1,2,3$.
        \item $\F(M)$ has the $3$-Schur property but fails the $c$-Schur property for $c<3$.
        \item $\F(M)$ fails property $(iii_c)$ from Proposition~\ref{P:schur} for $c<2$. In particular, it fails the $1$-strong Schur property.
     \end{enumerate} 
\end{example}

\begin{proof}
Consider the graph whose set of vertices is $\zet\setminus\{0\}$ and $m,n\in \zet\setminus\{0\}$ are joined by an edge if and only if $mn<0$ and $m\ne-n$.
Let $d$ denote the shortest-path distance. It is clearly given by
$$d(m,n)=\begin{cases}
  3, & m=-n,  \\ 1, & mn<0, m\ne-n,\\ 2, & mn>0, m\ne n,\\ 0, & m=n.
\end{cases}$$
Let $M=(\zet\setminus \{0\},d)$. Clearly $(b)$ is satisfied. Thus $\F(M)$ has the $3$-Schur property by Proposition~\ref{P:discrete-ab}.

To prove the remaining part of the statement we will use the following claim:

\smallskip

\noindent{\tt Claim:} If $f:M\to\er$ is a $1$-Lipschitz function, then there is at most one $n\in\en$ with 
$f(n)-f(-n)>1$.

\smallskip

Indeed, assume that $f$ is $1$-Lipschitz and $n\in\en$ is such that $f(n)-f(-n)>1$. Let $m\in\en\setminus\{n\}$ be arbitrary. Then
$$f(m)=f(-n)+(f(m)-f(-n))\le f(-n)+d(m,-n)=f(-n)+1< f(n)$$
and
$$f(-m)=f(n)+(f(-m)-f(n))\ge f(n)-d(-m,n)=f(n)-1.$$
Thus
$$f(m)-f(-m)< f(n)-(f(n)-1)=1,$$
which completes the proof of the claim.

\smallskip

Let us continue by proving that $\F(M)$ fails $c$-Schur property for $c<3$. Let us fix a distinguished point, say $1$. Let
$x_n=\delta(n)-\delta(-n)$ for each $n\in \en$. Then
$$\norm{x_n}=d(n,-n)=3$$
for each $n\in\en$. Thus $\limsup_n\norm{x_n}=3$. In view of Proposition~\ref{P:schur} the proof will be complete if we show that each weak$^*$-cluster point of $(x_n)$ in $\F(M)^{**}$ has norm at most $1$.

Assume the contrary, i.e., that there is $x^{**}$, a weak$^*$-cluster point of $(x_n)$, with $\norm{x^{**}}>1$. 
Then there is a norm-one element $f\in \F(M)^*$ with $\ip{x^{**}}{f}>1$. Hence $\ip{f}{x_n}>1$ for infinitely many $n\in\en$. It means that $f$ is a $1$-Lipschitz function on $M$ such that
$$f(n)-f(-n)>1\mbox{ for infinitely many }n\in\en.$$
But this contradicts the above claim.

\smallskip

Finally, let us prove assertion $(c)$. We take the sequence $(x_n)$ as above. Observe that $\norm{x_n-x_m}=4$ whenever $m\ne n$. Indeed, 
$$\begin{aligned}
    \norm{x_n-x_m}&=\norm{\delta(n)-\delta(-n)-\delta(m)+\delta(-m)}\\&\le \norm{\delta(n)-\delta(m)}+\norm{\delta(-m)-\delta(-n)}=d(n,m)+d(-m,-n)=4,\end{aligned}$$
which proves one inequality. To prove the converse, observe that the formula
$$g(n)=g(-m)=1,\ g(-n)=g(m)=-1$$
is $1$-Lipschitz on $\{m,n,-m,-n\}$ and so it may be extended to a $1$-Lipschitz function $\widetilde{g}$ on $M$. 
Then $h=\widetilde{g}-\widetilde{g}(1)$ is a $1$-Lipschitz function from $\Lip_0(M)$ and
$$\ip{h}{x_n-x_m}=h(n)-h(-n)-h(m)+h(-m)=g(n)-g(-n)-g(m)+g(-m)=4.$$

So, $(x_n)$ is $4$-discrete, in particular $\wca{x_n}=4$. Let us estimate $\de{x_n}$. Let $f\in \Lip_0(M)$ be $1$-Lipschitz. Applying the above claim to $f$ and $-f$ we see that for all $n\in\en$ with at most two exceptions we have $\abs{f(n)-f(-n)}\le 1$. In particular, there is $n_0\in\en$ such that $\abs{f(n)-f(-n)}\le 1$ for $n\ge n_0$.
Thus for $m,n\ge n_0$ we have
$$\begin{aligned}    
\abs{\ip{f}{x_n-x_m}}&=\abs{f(n)-f(-n)-f(m)+f(-m)}\\&\le\abs{f(n)-f(-n)}+\abs{f(m)-f(-m)}\le 2.\end{aligned}$$
Thus $\ca{\ip{f}{x_n}}\le 2$. Since $f$ was arbitrary, we deduce $\de{x_n}\le2$, which completes the proof.    
\end{proof}

\section{Final remarks and open problems}

In this section we collect open problems which arise from the results. Some of them concern optimality of constants and inequalities.

\begin{ques}
    Is implication $(iv_c)\to(i_{2c})$ in Proposition~\ref{P:wsc} optimal?
\end{ques}

This question is related to comparison of two natural notion of quantitative weak sequential completeness.
We note that we know no example showing that factor $2$ is necessary. In fact, we do not know whether some factor greater than $1$ is needed. Te next question is similar.

\begin{ques}
    Is implication $(iii_c)\to(ii_{2c+1})$ in Proposition~\ref{P:schur} optimal?
\end{ques}

This question is related to comparison of two natural notions of quantitative Schur property. In this case some increase of constant is necessary, as in Examples~\ref{ex:LF} and~\ref{ex:LF-complex} we provide a Banach space satisfying $(iii_1)$, $(ii_2)$, but not $(ii_c)$ for $c<2$. But it is not clear whether the optimal constant is $2c+1$, $2c$ or $c+1$ (or something else).

A further question is related to the sufficient condition from Section~\ref{sec:m1}.

\begin{ques}
    Is the condition from Proposition~\ref{P:subs-1Schur} necessary for the $1$-Schur property of a real Banach space?
\end{ques}

We note that for complex spaces the condition is not necessary by Remark~\ref{rem:m1}(3).

The last question is devoted to Lipschitz-free spaces.

\begin{ques}
    Let $M$ be a uniformly separated metric space.
    \begin{enumerate}[$(1)$]
        \item Does $\F(M)$ has the quantitative Schur property? Does it have the $3$-Schur property?
        \item Is $\F(M)$ quantitatively weakly sequentially complete?
        \item Are the two measures of non-compactness ($\omega(\cdot)$ an $\wk{\cdot}$ equivalent in $\F(M)$?
     \end{enumerate}
\end{ques}

As remarked above, due to \cite{p1u} we know that $\F(M)$ has the Schur property. If $M$ is additionally bounded, we know it has quantitative Schur property by Proposition~\ref{P:discrete-ab}. But the key question is whether the constant $\frac ba$ in the quoted proposition is optimal. Examples~\ref{ex:LF} and~\ref{ex:LF-3} show that the constant is optimal if $\frac ba$ equals $2$ or $3$. But it is not clear how to construct examples with worse constants.

\bibliographystyle{acm}
\bibliography{schurmiry}
\end{document}